\documentclass[10pt]{amsart}

\usepackage{amsmath,amsthm,amsfonts,latexsym,amssymb,mathrsfs,setspace,enumerate,color,cite,graphicx,epsf,mathtools,fullpage,array,amstext,verbatim}

\newtheorem{thm}{Theorem}
\newtheorem{lem}[thm]{Lemma}
\newtheorem{prop}[thm]{Proposition}
\newtheorem{cor}[thm]{Corollary}
\numberwithin{equation}{section}
\numberwithin{thm}{section}

\newtheorem{conj}[thm]{Conjecture}
\newtheorem{question}[thm]{Question}

\newcommand{\rat}{\mathbb Q}
\newcommand{\real}{\mathbb R}
\newcommand{\com}{\mathbb C}
\newcommand{\alg}{\overline\rat}
\newcommand{\algt}{\alg^{\times}}

\newcommand{\intg}{\mathbb Z}
\newcommand{\nat}{\mathbb N}

\newcommand{\xx}{{\bf x}}
\newcommand{\yy}{{\bf y}}
\newcommand{\zz}{{\bf z}}

\newcommand{\tors}{\mathrm{tors}}

\newcolumntype{P}[1]{>{\centering\arraybackslash}p{#1}}

\title[Exceptional Points]{Counting Exceptional Points for Rational Numbers Associated to the Fibonacci Sequence}
\author{Charles L. Samuels}

\begin{document}

\begin{abstract}
	If $\alpha$ is a non-zero algebraic number, we let $m(\alpha)$ denote the Mahler measure of the minimal polynomial of $\alpha$ over $\intg$.
	A series of articles by Dubickas and Smyth, and later by the author, develop a modified version of the Mahler measure called the $t$-metric Mahler measure, denoted $m_t(\alpha)$.
	For fixed $\alpha\in \alg$, the map $t\mapsto m_t(\alpha)$ is continuous, and moreover, is infinitely differentiable at all but finitely many points, called {\it exceptional points} for $\alpha$.  
	It remains open to determine whether there is a sequence of elements $\alpha_n\in \alg$ such that the number of exceptional points for $\alpha_n$ tends to $\infty$ as $n\to \infty$.
	
	We utilize a connection with the Fibonacci sequence to formulate a conjecture on the $t$-metric Mahler measures.  If the conjecture is true, we prove that it is best possible and that it implies the
	the existence of rational numbers with as many exceptional points as we like.  Finally, with some computational assistance, we resolve various special cases of the conjecture that constitute 
	improvements to earlier results.
\end{abstract}

\maketitle

\section{Introduction} \label{Intro}

Suppose $\alpha$ is a non-zero algebraic number with minimal polynomial over $\intg$ given by
\begin{equation*}
	F(z) = a\cdot \prod_{i=1}^d (z - \alpha_i).
\end{equation*}
Under these assumptions, the {\it (logarithmic) Mahler measure} of $\alpha$ is defined to be 
\begin{equation*}
	m(\alpha) = \log |a| + \sum_{i=1}^d \log\max\{1,|\alpha_i|\}.
\end{equation*}
It is obvious from the definition that $m(\alpha) \geq 0$ for all $\alpha\in \algt$, and moreover, it follows from Kronecker's Theorem \cite{Kronecker} that $m(\alpha) = 0$ if and only if $\alpha$ is a root of unity.
We also note that the behavior of $m(\alpha)$ is particularly straightforward when $\alpha\in \rat^\times$.  Indeed, if $\alpha = r/s$ and $\gcd(r,s) = 1$ then $m(\alpha) = \log\max\{|r|,|s|\}$.

In attempting to construct large prime numbers, D.H. Lehmer \cite{Lehmer} came across the problem of determining whether there exists a sequence of non-zero
algebraic numbers $\{\alpha_n\}$, not roots of unity, such that $m(\alpha_n)$ tends to $0$ as $n\to\infty$.  This problem remains unresolved, although substantial evidence
suggests that no such sequence exists (see \cite{BDM, MossWeb, Schinzel, Smyth}, for instance).  This assertion is typically called Lehmer's conjecture.

\begin{conj}[Lehmer's Conjecture]
	There exists $c>0$ such that $m(\alpha) \geq c$ whenever $\alpha\in \algt$ is not a root of unity. 
\end{conj}

Dobrowolski \cite{Dobrowolski} provided the best known lower bound on $m(\alpha)$ in terms of $\deg\alpha$, while Voutier \cite{Voutier}
later gave a  version of this result with an effective constant.  Nevertheless, only little progress has been made on Lehmer's conjecture for an arbitrary algebraic number $\alpha$.

Dubickas and Smyth \cite{DubSmyth, DubSmyth2} were the first to study a modified version of the Mahler measure which gives rise to a metric on $\algt/\algt_\tors$.
A point $(\alpha_1,\alpha_2,\ldots,\alpha_N)\in (\algt)^N$ is called a {\it product representation of $\alpha$} if $\alpha = \prod_{n=1}^N \alpha_n$,
and we write $\mathcal P(\alpha)$ to denote the set of all product representations of $\alpha$.  Dubickas and Smyth defined the {\it metric Mahler measure} by
\begin{equation} \label{m1}
	m_1(\alpha) = \inf\left\{ \sum_{n=1}^N m(\alpha_n): (\alpha_1,\alpha_2,\ldots,\alpha_N)\in \mathcal P(\alpha)\right\}.
\end{equation}
It is verified in \cite{DubSmyth2} that $m_1:\algt\to [0,\infty)$ satisfies the following key properties:
\begin{enumerate}[(i)]
	\item\label{Mod} $m_1(\alpha) = m_1(\zeta\alpha)$ for all $\alpha\in \algt$ and $\zeta\in \algt_\tors$
	\item\label{Inverse} $m_1(\alpha) = m_1(\alpha^{-1})$ for all $\alpha \in \algt$
	\item\label{Tri} $m_1(\alpha\beta) \leq m_1(\alpha) + m_1(\beta)$ for all $\alpha,\beta\in \algt$.
\end{enumerate}
These facts combine to ensure that $(\alpha,\beta) \mapsto m_1(\alpha\beta^{-1})$ is a well-defined metric on $\algt/\algt_\tors$ which
induces the discrete topology if and only if Lehmer's conjecture is true.

The author \cite{SamuelsCollection, SamuelsParametrized,SamuelsMetric} extended the metric Mahler measure to form a parametrized family of metric Mahler measures.
If $\bar\alpha = (\alpha_1,\alpha_2,\ldots,\alpha_N)\in \mathcal P(\alpha)$ then we define the {\it measure function of $\bar\alpha$} to be the map $f_{\bar\alpha}:(0,\infty) \to [0,\infty)$
given by
\begin{equation*}
	f_{\bar\alpha}(t) = \left( \sum_{n=1}^N m(\alpha_n)^t\right)^{1/t}.
\end{equation*}
The {\it $t$-metric Mahler measure} of $\alpha$ is defined to be
\begin{equation*} \label{mt}
	m_t(\alpha) = \inf\left\{ f_{\bar\alpha}(t): \bar\alpha \in \mathcal P(\alpha) \right\}
\end{equation*}
and we note that $m_1(\alpha)$ agrees with the definition provided by Dubickas and Smyth in \eqref{m1}.  Properties \eqref{Mod} and \eqref{Inverse} continue to hold with $m_t$ in place of $m_1$,
however, the analog of \eqref{Tri} is that $$m_t(\alpha\beta)^t \leq m_t(\alpha)^t + m_t(\beta)^t$$ for all $\alpha,\beta\in \algt$ and all $t > 0$.  As a result, $(\alpha,\beta) \mapsto m_t(\alpha\beta^{-1})^t$ 
defines a metric on $\algt/\algt_\tors$ which induces the discrete topology if and only if Lehmer's conjecture is true.

The definition of $m_t(\alpha)$ requires examining the infinite collection $\mathcal P(\alpha)$, however, the main result of \cite{SamuelsMetric} gives us hope for a dramatic simplification.

\begin{thm}\label{FiniteInfimum}
	If $\alpha$ is an algebraic number then there exists a finite set $\mathcal X\subseteq \mathcal P(\alpha)$ such that $m_t(\alpha) =  \min\left\{ f_{\bar\alpha}(t): \bar\alpha \in \mathcal X \right\}$
	for all $t > 0$.
\end{thm}

Although Theorem \ref{FiniteInfimum} certainly implies that the infimum in $m_t(\alpha)$ is attained for all $t$ (an assertion that the author proved earlier in \cite{SamuelsInfimum}), its primary
value is that the infimum attaining points may all be chosen from a finite set which is independent of $t$.  Nevertheless, we caution the reader that the proof of Theorem \ref{FiniteInfimum} provides no 
method for determining a particular set $\mathcal X$ which satisfies its conclusion, and in general, it remains open to provide formula for such a set in terms of $\alpha$. 
For further study of this vague problem, it will be useful to provide two additional definitions.

\begin{enumerate}[(i)]
	\item We say that a positive real number $t$ is {\it standard for $\alpha$} if there exists $\bar\alpha\in \mathcal P(\alpha)$ and an open neighborhood $U$ of $t$ such that
		$m_t(\alpha) = f_{\bar\alpha}(t)$ for all $t\in U$.  
	\item Any point which fails to be standard for $\alpha$ is called {\it exceptional for $\alpha$}.
\end{enumerate}

Roughly speaking, the standard points are those points where the map $t\mapsto m_t(\alpha)$ matches the behavior of a measure function, while the exceptional points are those where it differs.
Since the behavior of a measure function is easily understood (it is simply the norm of a vector with real entries), the map $t\mapsto m_t(\alpha)$ may only exhibit unusual behavior at an exceptional point.
For example, we established in \cite{SamuelsCollection} that $t$ is standard if and only if $t\mapsto m_t(\alpha)$ is infinitely differentiable at $t$.
It follows from Theorem \ref{FiniteInfimum} that exceptional points are rather sparse.

\begin{cor} \label{FiniteExceptional}
	Every algebraic number has finitely many exceptional points.
\end{cor}

Corollary \ref{FiniteExceptional} comes equipped with a similar caveat as Theorem \ref{FiniteInfimum}.  Although we know there are finitely many exceptional points, the proof of 
Corollary \ref{FiniteExceptional} provides no general strategy for listing those points, nor does it suggest a strategy for estimating how many such points there are.  
This discussion leads to the following motivating problem.

\begin{question} \label{ExceptionalCounting}
	For every integer $k\geq 0$ does there exist an algebraic number having $k$ exceptional points?
\end{question}

We shall address Question \ref{ExceptionalCounting} by considering a special case of rational numbers studied in \cite{SamuelsFibonacci}.  For this purpose, let
$\{h_i\}_{i=0}^\infty$ be the Fibonacci sequence defined so that $h_0 = 0$ and $h_1= 1$.  Further let $N\geq 3$ be an integer and select primes $p$ and $q$ such that
\begin{equation} \label{WeakCompatible}
	\frac{h_N}{h_{N-1}} < \frac{\log q}{\log p} < \frac{h_{N-1}}{h_{N-2}}\quad\mbox{or} \quad \frac{h_{N-1}}{h_{N-2}} < \frac{\log q}{\log p} < \frac{h_N}{h_{N-1}}.
\end{equation}
Using the fact that $[x,2x]$ contains a prime for all $x \geq 1$ (see \cite{Cheby}), it can be shown that expressions of the form $\log q/\log p$ are dense in $(0,\infty)$, and hence, we are certain
that there exist primes satisfying \eqref{WeakCompatible}.  Many future definitions in this article depend on the choices of $N$, $p$ and $q$.  However, in order to prevent our notation from becoming 
excessively cumbersome, we shall often suppress this dependency in that notation.  The only exception to this convention is Section \ref{FirstProofs} where we will need to be more cautious with our notation.

We define the linear transformation $A:\real^N\to \real^2$ using the $2\times N$ matrix
\begin{equation*}
		A = \left( \begin{array}{cccc} h_1 & h_2 & \cdots & h_{N} \\ h_0 & h_1 & \cdots & h_{N-1} \end{array} \right).
\end{equation*}
It is easily verified that the rows of $A$ are linearly independent over $\real$, which implies that $A$ is a surjection and $\dim_\real(\ker A) = N-2$.
We also write $\nat_0^N = \{(x_1,x_2,\ldots,x_N)^T\in \real^N: x_i\in \intg,\ x_i \geq 0\}$, and 
if $n$ is an integer with $1\leq n\leq N$, then we define
\begin{equation} \label{VDef}
	\mathcal V_n = \left\{ \xx\in \nat_0^N: A\xx = \begin{pmatrix} h_n \\ h_{n-1} \end{pmatrix} \right\}.
\end{equation}
The elements of $\mathcal V_n$ are technically column vectors, however, for ease of notation, we shall often write them as row vectors.
As we shall not discuss elements in the dual of $\real^N$ in this paper, this notation will not create any ambiguity.

Obviously $\mathcal V_n$ is finite, and if $(x_1,x_2,\ldots,x_N)\in \mathcal V_n$ then $x_i = 0$ for all $n< i \leq N$.  Therefore, while $\mathcal V_n$ 
certainly depends on $N$, replacing $N$ by different value on the right hand side of \eqref{VDef} while keeping $n$ fixed, we simply attach or remove a list of $0$'s from the tail of each point 
in $\mathcal V_n$.
Still assuming that $1\leq n\leq N$, we define $$\alpha_n = \frac{p^{h_n}}{q^{h_{n-1}}}.$$  Each point in $\mathcal V_n$ is associated to a product representation of $\alpha_n$
via the map $\omega:\mathcal V_n\to\mathcal P(\alpha_n)$ given by
\begin{equation*}
	\omega((x_1,x_2,\ldots,x_n)) = \left( \underbrace{ \frac{p^{h_1}}{q^{h_0}},\cdots,\frac{p^{h_1}}{q^{h_0}}}_{x_1\mbox{ times}},  
			 \underbrace{\frac{p^{h_2}}{q^{h_1}},\cdots,\frac{p^{h_2}}{q^{h_1}}}_{x_2\mbox{ times}}, \cdots\cdots,
			  \underbrace{\frac{p^{h_N}}{q^{h_{N-1}}},\cdots,\frac{p^{h_N}}{q^{h_{N-1}}}}_{x_N\mbox{ times}} \right).
\end{equation*}
The {\it measure function} of a point $\xx\in \mathcal V_n$ is simply defined to be the measure function of $\omega(\xx)$, and moreover, we shall write $f_\xx(t) = f_{\omega(\xx)}(t)$
for all $t>0$.  As a result, we obtain that
\begin{equation*}
	f_\xx(t) = \left(\sum_{i=1}^N x_i m\left(\frac{p^{h_i}}{q^{h_{i-1}}}\right)^t\right)^{1/t} = \left(\sum_{i=1}^n x_i m\left(\frac{p^{h_i}}{q^{h_{i-1}}}\right)^t\right)^{1/t},
\end{equation*}
where we deduce the second equality from our observation following \eqref{VDef}.  The main result of \cite{SamuelsFibonacci} shows that $m_t(\alpha_n)$ may be computed
by considering only points in $\mathcal V_n$.

\begin{thm} \label{PreviousMain}
	Suppose that $N\geq 3$ is an integer and $(p,q)$ is pair of primes satisfying \eqref{WeakCompatible}.  If $1\leq n\leq N$ then 
	$m_t(\alpha_n) = \min\{f_\xx(t):\xx\in \mathcal V_n\}$ for all $t > 0$.
\end{thm}

The significance of Theorem \ref{PreviousMain} is that it substantially restricts the collection of product representations we need to search in order to evaluate $m_t(\alpha_n)$.  Indeed,
there are product representations of $\alpha_n$ which use any particular integer power.  However, Theorem \ref{PreviousMain} shows that we need only consider those
which use exponent pairs of the form $(h_i,h_{i-1})$ for $1\leq i\leq N$.

Our goal for this article is to address Question \ref{ExceptionalCounting} by counting exceptional points for $\alpha_n$.
As part of this process, it will be useful to be able to replace $\mathcal V_n$ in Theorem \ref{PreviousMain} by a significantly smaller set.
In Section \ref{ConjecturedReplace}, we pose a conjecture (Conjecture \ref{MainConj}) identifying a particular set $\mathcal S_n$ which we believe satisfies
\begin{equation} \label{PreliminaryConj}
	m_t(\alpha_n) = \min\{f_\xx(t):\xx\in \mathcal S_n\}.
\end{equation}  
We show that if Conjecture \ref{MainConj} is true then it is best possible\footnote{{\it Best possible} means that the set $\mathcal S_n$ on the right hand side of \eqref{PreliminaryConj}
cannot be replaced with a smaller set while still maintaining equality.  See Theorem \ref{ConjectureConsequences}\eqref{Best} for the more rigorous version of this statement.}, 
and moreover, it resolves Question \ref{ExceptionalCounting} in the affirmative.
We utilize Section \ref{Progress} to discuss our progress in the direction of Conjecture \ref{MainConj} including various computational results which resolve the conjecture for $N\leq 13$.
As part of that progress, we show that Conjecture \ref{MainConj} may be reduced to the study of a particular subset of $\mathcal V_n$.  This discussion relates Question \ref{ExceptionalCounting}
to several problems on the behavior of the Fibonacci Sequence.  Following these discussions, we provide the proofs of all results in the subsequent three sections.


\section{Conjectured Replacement for $\mathcal V_n$.} \label{ConjecturedReplace}

For the purposes of this section, we remind the reader that all definitions depend on the choices of $N$ and $(p,q)$ even though we shall often suppress this dependency in our notation.
As noted in the previous section, we shall define a particular subset of $\mathcal V_n$ and conjecture that this subset can replace $\mathcal V_n$ in Theorem \ref{PreviousMain}.
Before we can do so, we will need to impose an additional restriction on the pair of primes $(p,q)$ beyond that which appears in \eqref{WeakCompatible}.  
This discussion begins with the following preliminary observation.

\begin{prop} \label{UniqueIntersection}
	Suppose that $n\geq 3$ then there exists a unique positive real number $t$ such that
	\begin{equation*}
		m\left( \frac{p^{h_n}}{q^{h_{n-1}}}\right)^t = m\left( \frac{p^{h_{n-1}}}{q^{h_{n-2}}}\right)^t  + m\left( \frac{p^{h_{n-2}}}{q^{h_{n-3}}}\right)^t,
	\end{equation*}
	and moreover, $t \geq 1$.
\end{prop}	

We shall write $t_n$ to denote the value of $t$ described in the conclusion of Proposition \ref{UniqueIntersection}.  
For an integer $N\geq 3$, we say that the ordered pair of primes $(p,q)$ is {\it compatible with $N$} if it satisfies \eqref{WeakCompatible} and
\begin{equation*}
	t_{N+1} < t_N < t_{N-1} < \cdots < t_4 < t_3.
\end{equation*}
Since this definition is rather exotic, we might be concerned that there exists $N\geq 3$ for which there is no compatible pair of primes.  Luckily, our next
result alleviates these concerns.

\begin{thm} \label{IntersectionWeaving}
	Suppose that $N\in \intg$ is such that $N\geq 3$.  There exists $\delta > 0$ such that if $p$ and $q$ are primes satisfying
	\begin{equation*}
		\left| \frac{\log q}{\log p} - \frac{1+\sqrt 5}{2}\right| < \delta
	\end{equation*}
	then $(p,q)$ is compatible with $N$.
\end{thm}

We recall that expressions of the form $\log q/\log p$ are dense in $(0,\infty)$.   Consequently, we know that for every $N\geq 3$, there exist infinitely many pairs of primes $(p,q)$ 
which are compatible with $N$.  Moreover, Theorem \ref{IntersectionWeaving} shows that we may locate such pairs of primes by looking near the golden ratio.

We now define a new set $\mathcal S_n \subseteq \mathcal V_n$ and we shall conjecture that Theorem \ref{PreviousMain} still holds even if $\mathcal V_n$ is replaced by
$\mathcal S_n$ in its statement.  For the purposes of this discussion, if $\xx = (x_1,x_2,\ldots,x_N)\in \mathcal V_n$ is such that $x_i = 0$ for all $i > k$, then we shall simply
write $\xx = (x_1,x_2,\ldots,x_k)$.  In particular, we may always write $\xx = (x_1,x_2,\ldots,x_n)$.  
We note that that this change of notation does not create ambiguity regarding the values of the measure functions $f_\xx(t)$ for any $t > 0$.

If $i\in \intg$ is such that $3\leq i\leq n+1$ we let
\begin{equation*}
	\xx_n(i) =  (\underbrace{0,0,\ldots,0,0}_{i-3\mbox{ times}},h_{n+1-i},h_{n+2-i}),
\end{equation*}
and for $2\leq n\leq N$, we define $\mathcal S_n = \{\xx_n(i): 3\leq i\leq n+1\}$.  To extend this definition to $n=1$ we define $\xx_1(2) = (1)$ and write $\mathcal S_1 = \{\xx_1(2)\}$.
It is easily verified from the definition that $$\xx_n(n+1) = (\underbrace{0,0,\ldots,0,0}_{n-1\mbox{ times}},1) \in \mathcal V_n$$ for all $n$, and we shall call this point the {\it trivial element} of $\mathcal V_n$.
From these observations, we conclude that that $\mathcal S_1\subseteq \mathcal V_1$ and $\mathcal S_2\subseteq \mathcal V_2$.
By applying the the recurrence relation from the Fibonacci Sequence, we are also able to obtain that
\begin{equation} \label{StandardReducible}
	\xx_n(i) = \xx_{n-1}(i) + \xx_{n-2}(i)\quad\mbox{ for all } 3\leq i\leq n-1
\end{equation}
and  
\begin{equation} \label{SpecialReducible}
	\xx_n(n) = \xx_{n-1}(n) + \xx_{n-2}(n-1).
\end{equation}
Using induction on $n$, these observations combine to ensure that $\mathcal S_n\subseteq\mathcal V_n$ for all $1\leq n\leq N$.  Additionally, we find it worth noting that $\#S_n = n-1$ for all $n$.

The definition of $\mathcal S_n$ makes this set appear more complicated than it actually is so we shall provide an example which we believe provides clarification.  Taking $N= 7$ and $n=5$ then
the vectors $\xx_5(i)$ are given by
\begin{equation*}
	x_5(3) = \begin{pmatrix} 2 \\ 3 \\ 0 \\0 \\ 0  \end{pmatrix}, \ x_5(4) = \begin{pmatrix} 0 \\ 1 \\ 2 \\0 \\0  \end{pmatrix}, \
	x_5(5) = \begin{pmatrix} 0 \\ 0 \\ 1 \\1 \\0 \end{pmatrix}, \ x_5(6) = \begin{pmatrix} 0 \\ 0 \\ 0 \\0 \\1  \end{pmatrix}
\end{equation*}
so that
\begin{equation*}
	\mathcal S_5 = \left\{  \begin{pmatrix} 2 \\ 3 \\ 0 \\0 \\ 0 \end{pmatrix}, \begin{pmatrix} 0 \\ 1 \\ 2 \\0 \\0  \end{pmatrix}, 
	\begin{pmatrix} 0 \\ 0 \\ 1 \\1 \\0 \end{pmatrix}, \begin{pmatrix} 0 \\ 0 \\ 0 \\0 \\1  \end{pmatrix} \right\}.
\end{equation*}
In a similar manner we obtain that
\begin{equation*}
	\mathcal S_6 = \left\{  \begin{pmatrix} 3 \\ 5 \\ 0 \\0 \\ 0 \\ 0 \end{pmatrix}, \begin{pmatrix} 0 \\ 2 \\ 3 \\0 \\0 \\ 0 \end{pmatrix}, 
	\begin{pmatrix} 0 \\ 0 \\ 1 \\2 \\0\\ 0 \end{pmatrix}, \begin{pmatrix} 0 \\ 0 \\ 0 \\1 \\1\\ 0  \end{pmatrix},
	\begin{pmatrix} 0 \\ 0 \\ 0 \\0 \\0\\ 1   \end{pmatrix} \right\}
\end{equation*}
and
\begin{equation*}
	\mathcal S_7 = \left\{  \begin{pmatrix} 5 \\ 8 \\ 0 \\0 \\ 0 \\ 0 \\ 0 \end{pmatrix}, \begin{pmatrix} 0 \\ 3 \\ 5 \\0 \\0 \\ 0 \\ 0 \end{pmatrix}, 
	\begin{pmatrix} 0 \\ 0 \\ 2 \\3 \\0\\ 0 \\ 0 \end{pmatrix}, \begin{pmatrix} 0 \\ 0 \\ 0 \\1 \\2\\ 0 \\ 0  \end{pmatrix},
	\begin{pmatrix} 0 \\ 0 \\ 0 \\0 \\1\\ 1 \\ 0  \end{pmatrix}, \begin{pmatrix} 0 \\ 0 \\ 0 \\0 \\0\\ 0 \\ 1  \end{pmatrix} \right\}.
\end{equation*}
We remind the reader that all of the above vectors are $7$-dimensional since we have chosen $N=7$.  As our notation permits, we have often omitted $0$'s and the end of each vector.
As promised, we believe that $\mathcal S_n$ may replace $\mathcal V_n$ in Theorem \ref{PreviousMain}.

\begin{conj} \label{MainConj}
	Suppose that $N\geq 3$ is an integer and $(p,q)$ is pair of primes which is compatible with $N$.
	If $n$ is an integer with $1\leq n\leq N$ then $m_t(\alpha_n) = \min\{f_\xx(t):\xx\in \mathcal S_n\}$ for all $t > 0$.
\end{conj}

We shall discuss our progress in the direction of Conjecture \ref{MainConj} in Section \ref{Progress}.
For now, we assert that if Conjecture \ref{MainConj} is correct then it is both best possible and it resolves Question \ref{ExceptionalCounting} in the affirmative.

\begin{thm} \label{ConjectureConsequences}
	Suppose that $N\geq 3$ is an integer and $(p,q)$ is pair of primes compatible with $N$.  If $1\leq n \leq N$ is such that
	$m_t(\alpha_n) = \min\{f_\xx(t):\xx\in \mathcal S_n\}$ for all $t > 0$ then the following conditions hold.
	\begin{enumerate}[(i)]
		\item\label{Best} If $\mathcal X\subseteq \mathcal V_n$ is such that $m_t(\alpha_n) = \min\{f_\xx(t):\xx\in \mathcal X\}$ then $\mathcal S_n \subseteq \mathcal X$.
		\item $\{t_3,t_4,\ldots,t_{n-1},t_n\}$ are the exceptional points for $\alpha_n$.  In particular, if $2\leq n\leq N$ then $\alpha_n$ has precisely $n-2$ exceptional points.
	\end{enumerate}
\end{thm}

The second statement of Theorem \ref{ConjectureConsequences} would indeed resolve Question \ref{ExceptionalCounting} in the affirmative.  After all, if we wished to 
create a rational number having $k$ exceptional points, we could apply Theorem \ref{IntersectionWeaving} to obtain a pair of primes $(p,q)$ which is compatible with $k+2$.
Then by Theorem \ref{ConjectureConsequences}, under the assumption of Conjecture \ref{MainConj}, we would obtain that $\alpha_{k+2}$ has $k$ exceptional points.


\section{Progress toward Conjecture \ref{MainConj}} \label{Progress}

Since it can be easily checked that $\mathcal V_1 = \mathcal S_1$ and $\mathcal V_2 = \mathcal S_2$, Conjecture \ref{MainConj} holds in the cases where $n\in \{1,2\}$.
Hence, it seems reasonable to attempt a proof by induction.  As we shall see in this section, this can be done for certain special cases of $N$, $p$ and $q$, but there is an obstruction
which prevents this method from being further generalized.
To demonstrate this progress as well as the obstruction, we must define two relevant sets in addition to $\mathcal V_n$ and $\mathcal S_n$ defined earlier.
For each of the subsequent definitions, we assume that $\zz = (z_1,z_2,\ldots,z_N) \in \mathcal V_n$.

We say that $\zz$ is {\it almost consecutive-free} if $z_j z_{j+1} \ne 0$ implies that $z_i = 0$ for all $i > j+1$.  We write $\mathcal C_n$ to denote the set of all almost consecutive-free
elements in $\mathcal V_n$.  It is obvious from the definition that $\mathcal S_n\subseteq \mathcal C_n$, but as our examples below will demonstrate, we do not have set equality.

Supposing that $\xx_1,\xx_2,\ldots,\xx_K \in \cup_{i=1}^n\mathcal V_i$, the $K$-tuple $(\xx_1,\xx_2,\ldots,\xx_K)$ is called a {\it factorization of $\zz$} if
\begin{equation*}
	\zz = \sum_{k=1}^K \xx_k.
\end{equation*}
Of course, we shall treat two factorizations as equivalent if one is simply a permutation of the other.  The point $\xx_n(n+1)\in \mathcal V_n$ has exactly one factorization, namely $(\xx_n(n+1))$.  
All other elements $\zz = (z_1,z_2,\ldots,z_N)\in \mathcal V_n$ have at least two factorizations obtained by examining the sums
\begin{equation} \label{TrivialFactorizations}
	\zz = \sum_{i=1}^1\zz\quad\mbox{and}\quad \zz = \sum_{i=1}^N z_i\xx_i(i+1).
\end{equation}
The left hand factorization in \eqref{TrivialFactorizations} is called the {\it trivial factorization of $\zz$} and the right hand factorization is called the {\it improper factorization of $\zz$}.
We say that the factorization $(\xx_1,\xx_2,\ldots,\xx_K)$ of $\zz$ is an {\it $\mathcal S$-type factorization} if $\xx_k \in \cup_{i=1}^n \mathcal S_i$ for all $1\leq k\leq K$.
 An element is called {\it $\mathcal S$-restricted} if all of its non-trivial factorizations are $\mathcal S$-type, and we write $\mathcal R_n$ to denote
the set of all $\mathcal S$-restricted elements of $\mathcal C_n$.

Although many of our earlier definitions in the paper depended on the primes $p$ and $q$, we note the sets $\mathcal C_n$ and $\mathcal R_n$ have no such dependency.
Strictly speaking, they do depend on $N$, however, any change in $N$ while keeping $n$ fixed will simply add or remove a list of $0$'s at the end of each element.  We now provide
an improvement over Theorem \ref{PreviousMain} which enables our progress toward Conjecture \ref{MainConj}.

\begin{thm} \label{MainProgress}
	If $n$ and $N$ are positive integers such that $1\leq n\leq N$ then 
	\begin{equation} \label{ProgressContainments}
		\mathcal S_n \subseteq \mathcal R_n \subseteq \mathcal C_n \subseteq \mathcal V_n.
	\end{equation}
	Moreover, if $(p,q)$ is a pair of primes compatible with $N$, then the following conditions hold:
	\begin{enumerate}[(i)]
		\item\label{Consecutive} $m_t(\alpha_n) = \min\{f_\xx(t):\xx\in \mathcal C_n\}$
		\item\label{Restricted} If $m_t(\alpha_i) = \min\{f_\xx(t):\xx\in \mathcal S_i\}$ for all $1 \leq i < n$ then $m_t(\alpha_n) = \min\{f_\xx(t):\xx\in \mathcal R_n\}$.
	\end{enumerate}
\end{thm}

Theorem \ref{MainProgress} constitutes an improvement over Theorem \ref{PreviousMain}, and moreover, it provides a further improvement if we are willing to assume Conjecture \ref{MainConj}
for all indices strictly smaller than $n$.  It is worth noting that we only rarely have $\mathcal R_n = \mathcal S_n$, and hence, Theorem \ref{MainProgress} falls short of a proof of Conjecture \ref{MainConj}.
As of this moment, we do not the believe that a minor improvement to the proof of Theorem \ref{MainProgress} is sufficient to obtain Conjecture \ref{MainConj}.  Among other things, the proof requires providing 
a proper non-trivial factorization for elements of $\mathcal V_n\setminus \mathcal S_n$, and as we shall see in the examples below, such a factorization does not always exist.

In spite of these shortcomings, Theorem \ref{MainProgress} can be used to establish special cases of Conjecture \ref{MainConj}.  If we can determine the points in $\mathcal V_n$, then it is a simple computational 
exercise to search those points to find those which lie in $\mathcal C_n$.  Once $\mathcal C_n$ is determined, then we may apply the following lemma to inductively list the points in $\mathcal R_n$.

\begin{lem} \label{RestrictedComputations}
	Suppose that $n$ and $N$ are positive integers such that $1\leq n\leq N$ and $\zz\in \mathcal C_n\setminus \mathcal S_n$.  Then $\zz\not\in \mathcal R_n$ if and only if
	there exist $1\leq i < n$ and $\xx\in \mathcal R_i\setminus S_i$ such that all entries of $\zz - \xx$ are non-negative.
\end{lem}

In view of Lemma \ref{RestrictedComputations}, we may apply the following four step process to list the points in $\mathcal R_n\setminus \mathcal S_n$:

\begin{enumerate}
	\item List the points in $\mathcal V_n$.
	\item Test each point in $\mathcal V_n$ to see whether it satisfies the required conditions to belong to $\mathcal C_n$.
	\item Form the sets $\mathcal C_n\setminus \mathcal S_n$.
	\item Assuming we have already found the points in $\mathcal R_i\setminus \mathcal S_i$ for all $ 1\leq i < n$, 
		use Lemma \ref{RestrictedComputations} to test each point in $\mathcal C_n\setminus \mathcal S_n$ for membership in $\mathcal R_n$.  This computation requires performing
	\begin{equation*}
		\#\mathcal (\mathcal C_n\setminus \mathcal S_n) \cdot \sum_{i=1}^{n-1} \# (\mathcal R_i\setminus \mathcal S_i) 
	\end{equation*}
	vector comparisons.
\end{enumerate}
	
Since $\mathcal V_n$ can be quite large compared to $n$, listing its points is a non-trivial computational problem.  Nevertheless, {\it Mathematica}'s {\tt Solve} command was sufficient to accomplish this
goal for $n\leq 13$.  Using the strategies outlined above, we have obtained complete lists of these sets when $n\leq 13$.  We shall provide additional details
in the discussion below, bur for now, we list their cardinalities.

\bigskip

\begin{center}
\begin{tabular}{| >{$}c<{$} | >{$}c<{$} | >{$}c<{$}  | >{$}c<{$}  |  >{$}c<{$}  |}
	\hline
	n & \#\mathcal V_n& \#\mathcal C_n &  \#\mathcal R_n  & \#\mathcal S_n  \\
	\hline \hline
	1 & 1 & 1 & 1  & 1  \\ \hline
	2 & 1 & 1 & 1 & 1    \\ \hline
	3 & 2 & 2 & 2& 2    \\ \hline
	4 & 3 & 3 & 3 & 3  \\ \hline
	5 & 6 & 4 & 4 &4  \\ \hline
	6 & 13 & 5 & 5 & 5  \\ \hline
	7 & 38 & 7 & 7 & 6\\ \hline
	8 & 139 & 11 & 8 & 7  \\ \hline
	9 & 695 & 20 & 10 & 8\\ \hline
	10 & 4,699 & 41  & 12 & 9  \\ \hline
	11 & 44, 359 & 104  & 18 & 10 \\ \hline
	12 & 589,359 & 310 & 24 &11 \\ \hline
	13 & 11,197,998 & 1101 & 44 & 12 \\ \hline
	\end{tabular}
\end{center}

\bigskip

One notable feature of this data is that $\mathcal C_n = \mathcal R_n = \mathcal S_n$ for all $n\leq 6$, so we immediately obtain Conjecture \ref{MainConj} for all $n\leq 6$ and all pairs of primes
$(p,q)$ which are compatible with $N$.  As a result, we obtain rational numbers having up to $4$ exceptional points without doing any more work.  If we wish to create more than $4$ exceptional points
using our method, we need to provide additional information regarding the sets $\mathcal R_n$.  Specifically, we need to examine each point in 
$\zz \in \mathcal R_i \setminus \mathcal S_i$, for all $1\leq i\leq n$, and show that 
\begin{equation} \label{MinTest}
	f_\zz(t) \geq \min\{f_\xx(t):\xx\in \mathcal S_i\}.
\end{equation}
We are able to accomplish this goal for $n\leq 13$ by calculating the points in $\mathcal R_n \setminus \mathcal S_n$ for all $n\leq 13$ and testing each one for inequality \eqref{MinTest}.

In order to abbreviate our reporting of the points in $\mathcal R_n \setminus \mathcal S_n$, we note that the map $\lambda:\mathcal \real^N\to \real^N$ given by
\begin{equation*}
	\lambda(((x_1,x_2,\ldots,x_N)) = (0,x_1,x_2,\ldots,x_{N-1})
\end{equation*}
defines an injection from $\mathcal V_{n-1}$ to $\mathcal V_n$ for any $2\leq n\leq N$.  Moreover, it can be shown that
\begin{equation*}
	\lambda(\mathcal R_{n-1}) \subseteq \mathcal R_n\quad\mbox{and}\quad \lambda(\mathcal R_{n-1}\setminus \mathcal S_{n-1}) \subseteq \mathcal R_n\setminus\mathcal S_n.
	\footnote{These set containments are still correct when $\mathcal \mathcal R_{n-1}$ and $\mathcal R_n$ are replaced with $\mathcal S_{n-1}$, $\mathcal C_{n-1}$ or $\mathcal V_{n-1}$
	and $\mathcal S_{n}$, $\mathcal C_{n}$ or $\mathcal V_{n}$, respectively.}
\end{equation*}
Therefore, when reporting the vectors in $\mathcal R_n \setminus \mathcal S_n$, it is sufficient to record only the points in
\begin{equation*}
	\Delta_n := \left(\mathcal R_n \setminus \mathcal S_n\right) \setminus \lambda(\mathcal R_{n-1} \setminus \mathcal S_{n-1}).
\end{equation*}
We caution the reader that we are currently unable to prove that
\begin{equation*}
	f_\zz(t) \geq \min\{f_\xx(t): \xx\in \mathcal S_{n-1}\} \implies f_{\lambda(\zz)}(t) \geq \min\{f_\xx(t): \xx\in \mathcal S_{n}\} \quad\mbox{for all } \zz\in \mathcal S_{n-1}.
\end{equation*}
Therefore, even though we shall only report points in $\Delta_n$, we must test all points of $\mathcal R_n \setminus \mathcal S_n$ for \eqref{MinTest} at each step rather than only those in $\Delta_n$.  
Our data regarding the points in $\Delta_n$ are in the table below.

\bigskip

\begin{center}
\begin{tabular}{| >{$}c<{$} | P{15cm}  |}
	\hline
	n & $\Delta_n$\\
	\hline \hline
	1 & None \\ \hline
	2 & None   \\ \hline
	3 & None   \\ \hline
	4 & None  \\ \hline
	5 & None  \\ \hline
	6 & None \\ \hline
	7 &    $(1,0,0,4)$                    \\ \hline
	8 &     None                  \\ \hline
	9 &     $(1,0,0,3,0,3)$                 \\ \hline
	10 &   $(1,0,0,2,0,6)$                  \\ \hline
	11 &   $(1,0,0,0,0,11),\ (1, 0, 0, 1, 0, 8, 0, 1),\ (1, 0, 0, 1, 0, 9, 1),\linebreak (1, 0, 0, 2, 0, 5, 0, 2),\ (1, 0, 0, 3, 0, 2, 0, 3)$  \\ \hline
	12 &    $(1,0,0,0,0,10,0,3),\ (1, 0, 0, 1, 0, 7, 0, 4),\ (1,0,0,1,0,9,0,0,2),\linebreak (1, 0, 0, 2, 0, 4, 0, 5),\ (1,0,0,3,0,1,0,6)$             \\ \hline
	13 &    $(1, 0, 0, 0, 0, 8, 0, 8),\ (1, 0, 0, 0, 0, 9, 0, 5,0,1),\  (1, 0, 0, 0, 0, 9, 0, 6, 1),\ (1, 0, 0, 0, 0,10, 0, 2, 0, 2),\linebreak
			 (1, 0, 0, 1, 0, 5, 0, 9),\ (1, 0, 0, 1, 0, 6, 0, 6, 0, 1),\ (1, 0, 0, 1, 0, 6, 0, 7, 1),\ (1, 0, 0, 1, 0, 7, 0, 3, 0, 2),\linebreak
			  (1, 0, 0, 1, 0, 8, 0, 0, 0, 3),\ (1, 0, 0, 2, 0, 2, 0, 10),\ (1, 0, 0, 2, 0, 3, 0, 7, 0, 1),\ (1, 0, 0, 2, 0, 3, 0, 8, 1), \linebreak
			  (1, 0, 0, 2, 0, 4, 0, 4, 0, 2),\ (1, 0, 0, 2, 0, 5, 0, 1, 0, 3),\ (1, 0, 0, 3, 0, 0, 0, 8, 0, 1),\ (1, 0, 0, 3, 0, 0, 0, 9, 1),\linebreak
			  (1, 0, 0, 3, 0, 0, 1, 10),\ (1, 0, 0, 3, 0, 1, 0, 5, 0, 2),\ (1, 0, 0, 3, 0, 2, 0, 2, 0, 3)$     \\ \hline
	\end{tabular}
\end{center}

\bigskip

Our methods for computing $\mathcal R_n$ seem to be insufficient for $n=14$.  Specifically, our methods require that we first compute $\mathcal V_n$ en route to computing $\mathcal C_n$ followed by
$\mathcal R_n$.  Hence, in order to provide more data, we would need to accomplish one of the following goals:
\begin{enumerate}
	\item Find a way to determine $\mathcal C_n$ or $\mathcal R_n$ without first listing the points in $\mathcal V_n$.
	\item Find a more efficient way to compute $\mathcal V_n$ than using {\it Mathematica}'s {\tt Solve} command.
\end{enumerate}
Even with improved computational methods, we don't believe the technique outlined in this section may be used to prove Conjecture \ref{MainConj}.  
For example, the only factorizations of $(1,0,0,4) \in \mathcal R_7$ are the trivial and improper factorizations, and neither such factorization provides assistance in proving \eqref{MinTest}.
When we provide the proofs of these results in Section \ref{MainProgressProofs}, we shall see explicitly why these factorizations are not useful.

Nevertheless, we are able to use the data provided above to solve Conjecture \ref{MainConj} when $n\leq 13$ by verifying \eqref{MinTest} for each relevant point.
For the purposes of this discussion, we shall take $N=13$ and $(p,q) = (1879,198301)$ and we verify using {\it Mathematica} that our choice of $(p,q)$ is compatible with $13$ (In fact, it can be shown that
$21$ is the largest integer with which these primes are compatible).    According to Theorem \ref{MainProgress}\eqref{Restricted}, to prove that $m_t(\alpha_7) = \min\{f_\xx(t):\xx\in \mathcal S_7\}$ 
it remains only to show that
\begin{equation*}
	f_{(1,0,0,4)}(t) \geq \min\{f_\xx(t):\xx\in \mathcal S_7\}.
\end{equation*}
We see the graphs of the relevant functions below.  Notice that the measure function for $(1,0,0,4)$ (the dashed curve) always lies above the minimum of the measure functions for points in 
$\mathcal S_7$ (the solid curves).

\begin{center}
	\includegraphics[height=8cm,width=13cm]{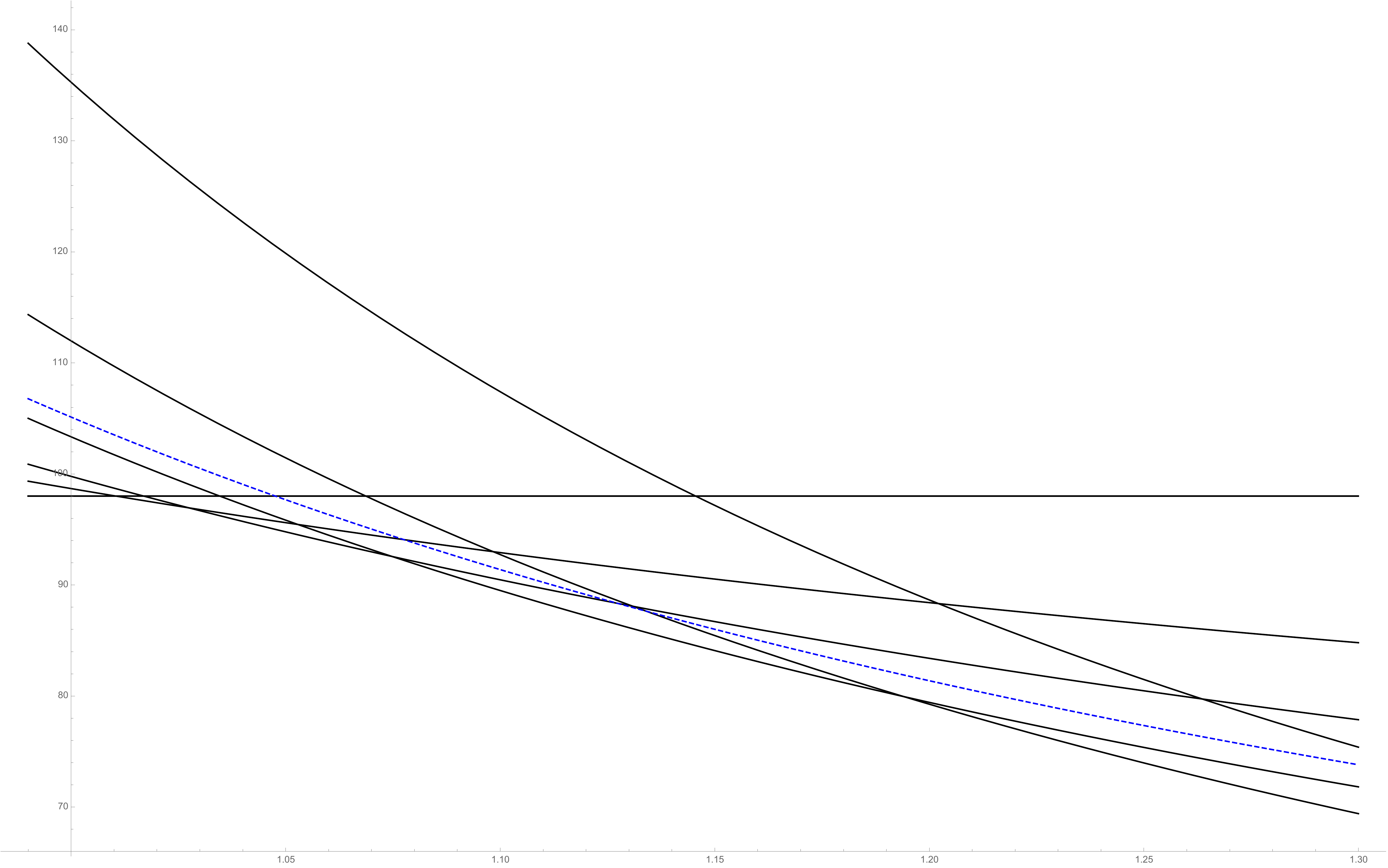}
\end{center}

As a result, we have now established that $m_t(\alpha_7) = \min\{f_\xx(t):\xx\in \mathcal S_7\}$, and in view of Theorem \ref{ConjectureConsequences}, we know that $\alpha_7$ has $5$ exceptional points.
The above diagram shows only $4$ exceptional points $t_7, t_6, t_5$ and $t_4$ while $t_3$ lies off the page.
Since we currently know the elements of $\mathcal R_n\setminus\mathcal S_n$ for $n\leq 13$, we can perform 
similar calculations when $n\in \mathcal\{8,9,10,11,12,13\}$ which lead us to a resolution of Conjecture \ref{MainConj} in these cases.  In particular, Theorem \ref{ConjectureConsequences}
establishes that
\begin{equation*}
	\alpha_{13} = \frac{1879^{233}}{198301^{144}}
\end{equation*}
has $11$ exceptional points.


\section{Proofs of Proposition \ref{UniqueIntersection} and Theorem \ref{IntersectionWeaving}} \label{FirstProofs}

The proof of Proposition \ref{UniqueIntersection} is very straightforward and we begin this subsection with its short proof.

\begin{proof}[Proof of Proposition \ref{UniqueIntersection}]
	We define $g:(0,\infty)\to (0,\infty)$ by
	\begin{equation*}
		g(t) = \left(m\left( \frac{p^{h_{n-1}}}{q^{h_{n-2}}}\right)^t  + m\left( \frac{p^{h_{n-2}}}{q^{h_{n-3}}}\right)^t\right)^{1/t}
	\end{equation*}
	We first observe that 
	\begin{equation*}
		\lim_{t\to 0^+} g(t) = \infty\quad\mbox{and}\quad \lim_{t\to\infty} g(t) = m\left( \frac{p^{h_{n-1}}}{q^{h_{n-2}}}\right),
	\end{equation*}
	so we certainly have that
	\begin{equation*}
		\lim_{t\to 0^+} g(t) > m\left( \frac{p^{h_n}}{q^{h_{n-1}}}\right) > \lim_{t\to\infty} g(t).
	\end{equation*}
	It can easily verified that $g$ is strictly decreasing so that the first statement of Proposition follows.  To see the second statement, we notice that
	\begin{align*}
		g(1) & = \max\{\log p^{h_{n-1}},\log q^{h_{n-2}}\} + \max\{\log p^{h_{n-2}},\log q^{h_{n-3}}\} \\
			& \geq \max\{\log p^{h_{n-1}} + \log p^{h_{n-2}},\log q^{h_{n-2}} + \log q^{h_{n-3}}\} \\
			& = \max\{\log p^{h_{n}},\log q^{h_{n-1}}\}.
	\end{align*}
	These observations yield
	\begin{equation*}
		g(1) \geq m\left( \frac{p^{h_n}}{q^{h_{n-1}}}\right)
	\end{equation*}
	and the result follows immediately.
\end{proof}

In Sections \ref{Intro} and \ref{ConjecturedReplace}, we noted that many of our definitions depended on particular choices of $N$, $p$ and $q$.  However, we often suppressed that dependency in order 
to prevent the notation from becoming excessively cumbersome.  Unfortunately, the most natural proof of Theorem \ref{IntersectionWeaving} studies 
the behavior of measure functions as $p$ and $q$ are chosen so that $\log q/\log p$ approaches the golden ratio.  As a result, we must employ more robust notation than we had previously used.

In view of these observations, we shall now write $t_n(p,q)$ to denote the unique positive real number such that
\begin{equation*}
	m\left( \frac{p^{h_n}}{q^{h_{n-1}}}\right)^{t_n(p,q)} = m\left( \frac{p^{h_{n-1}}}{q^{h_{n-2}}}\right)^{t_n(p,q)}   + m\left( \frac{p^{h_{n-2}}}{q^{h_{n-3}}}\right)^{t_n(p,q)}.
\end{equation*}
From Proposition \ref{UniqueIntersection} we know that $t_n(p,q) \geq 1$.  We must now consider an analog of Proposition \ref{UniqueIntersection} which does not depend on primes $p$ and $q$.
We let $\phi$ denote the golden ratio.

\begin{lem}\label{TargetUniqueIntersection}
	If $n\geq 3$ then the exists a unique positive real number $t$ such that
	\begin{equation*}
		\max\{h_n,\phi h_{n-1}\}^t = \max\{h_{n-1},\phi h_{n-2}\}^t + \max\{h_{n-2},\phi h_{n-3}\}^t,
	\end{equation*}
	and moreover, $t\geq 1$.
\end{lem}

The proof of Lemma \ref{TargetUniqueIntersection} is extremely similar to that of Proposition \ref{UniqueIntersection} so we need not include it here.  We shall now write $s_n$ to denote the
unique positive real number such that
\begin{equation} \label{sn}
	\max\{h_n,\phi h_{n-1}\}^{s_n}= \max\{h_{n-1},\phi h_{n-2}\}^{s_n} + \max\{h_{n-2},\phi h_{n-3}\}^{s_n},
\end{equation}
and note that $s_n\geq 1$.  In order to establish Theorem \ref{IntersectionWeaving}, we must prove that $t_n(p,q) > t_{n+1}(p,q)$ for all $3\leq n\leq N$ provided that $\log q/\log p$ is sufficiently
close to the golden ratio.  To this end, we shall first prove the following lemma.

\begin{lem}\label{TargetCompatible}
	$s_n > s_{n+1}$ for all $n\geq 3$.
\end{lem}
\begin{proof}
	We shall prove the lemma by contradiction so suppose that $n\geq 3$ is such that $s_n \leq s_{n+1}$ and consider two cases.
	
	{\bf Case 1}:  We assume first that $n$ is odd so that
	\begin{equation} \label{OddOrder}
		 \frac{h_{n-1}}{h_{n-2}} < \frac{h_{n+1}}{h_n} < \phi < \frac{h_n}{h_{n-1}} < \frac{h_{n-2}}{h_{n-3}}.
	\end{equation}
	where we utilize the convention that $h_1/h_0 = \infty$ so that these inequalities necessarily make sense.  In this situation, we apply the definitions of $s_n$ and $s_{n+1}$
	to obtain that
	\begin{equation*}
		h_n^{s_n} = (\phi h_{n-2})^{s_n} + h_{n-2}^{s_n}\quad\mbox{and}\quad (\phi h_n)^{s_{n+1}} = h_n^{s_{n+1}} + (\phi h_{n-2})^{s_{n+1}}.
	\end{equation*}
	Then setting $y_n = h_{n}/h_{n-2}$ we are lead to 
	\begin{equation*}
		\phi = \left( y_n^{s_n} - 1\right)^{1/s_n} \quad\mbox{and}\quad \phi = \frac{y_n}{(y_n^{s_{n+1}} - 1 )^{1/s_{n+1}}}
	\end{equation*}
	The function $t\mapsto (y_n^t - 1)^{1/t}$ is easily shown to be increasing, and therefore, we conclude that
	\begin{equation*}
		\phi \leq \frac{y_n}{(y_n^{s_{n}} - 1 )^{1/s_{n}}} = \frac{y_n}{\phi}.
	\end{equation*}
	These inequalities yield
	\begin{equation*}
		0 = \phi^2 - \phi - 1 \leq y_n - \phi - 1 = \frac{h_n}{h_{n-2}} - 1 - \phi = \frac{h_{n-1}}{h_{n-2}} - \phi.
	\end{equation*}
	By \eqref{OddOrder}, the right hand side of these inequalities is negative, a contradiction.
	
	{\bf Case 2}: We must now suppose that $n$ is even so that
	\begin{equation*}
		 \frac{h_{n-2}}{h_{n-3}} < \frac{h_{n}}{h_{n-1}} < \phi < \frac{h_{n+1}}{h_{n}} < \frac{h_{n-1}}{h_{n-2}}.
	\end{equation*}
	In this case, we must have that $h_{n-3} > 0$ so that the above inequalities make sense.  As in the previous case, we apply the definition of $s_n$ and $s_{n+1}$, but in this case
	we obtain
	\begin{equation*}
		(\phi h_{n-1})^{s_n} = h_{n-1}^{s_n} + (\phi h_{n-3})^{s_n}\quad\mbox{and}\quad h_{n+1}^{s_{n+1}} = (\phi h_{n-1})^{s_{n+1}} + h_{n-1}^{s_{n+1}}.
	\end{equation*}
	Isolating $h_{n-1}/h_{n-3}$ and $h_{n+1}/h_{n-1}$ we find that
	\begin{equation} \label{DoubleEquality}
		\frac{h_{n-1}}{h_{n-3}} = \left(1 - \phi^{-s_n}\right)^{-1/s_n}\quad \mbox{and}\quad \frac{h_{n+1}}{h_{n-1}} = \left(\phi^{s_{n+1}} + 1\right)^{1/s_{n+1}}.
	\end{equation}
	We now note that $t\mapsto (\phi^t + 1)^{1/t}$ is decreasing so that 
	\begin{equation*}
		\frac{h_{n+1}}{h_{n-1}} \leq  \left(\phi^{s_{n}} + 1\right)^{1/s_{n}}.
	\end{equation*}
	Next, we observe that $h_{n+1} = 3h_{n-1} - h_{n-3}$ and deduce that
	\begin{equation*}
		3 - \frac{h_{n-3}}{h_{n-1}} = \frac{3h_{n-1} - h_{n-3}}{h_{n-1}} = \frac{h_{n+1}}{h_{n-1}} \leq \left(\phi^{s_{n}} + 1\right)^{1/s_{n}}.
	\end{equation*}
	Now using the left hand equation of \eqref{DoubleEquality}, we obtain that 
	\begin{equation} \label{FakeInequality}
		3 \leq \left(1 - \phi^{-s_n}\right)^{1/s_n} + \left(\phi^{s_{n}} + 1\right)^{1/s_{n}}
	\end{equation}
	We can check that $s_n = 1$ provides equality in the inequality \eqref{FakeInequality}, and moreover, the right hand side of the inequality is strictly decreasing as a function of $s_n$.
	These assertions force $s_n =1$ and contradict the left hand equality of \eqref{DoubleEquality}.
\end{proof}

Our next goal is to show that $t_n(p,q)$ is as close as we like to $s_n$ provided that $\log q/\log p$ is sufficiently close to $\phi$.  This assertion, made rigorous in the following lemma, mostly 
completes the proof of Theorem \ref{IntersectionWeaving}.

\begin{lem}\label{TApprox}
	Let $\varepsilon > 0$ and $n\geq 3$ be an integer.  There exists $\delta > 0$ such that if $|\log q/\log p - \phi| < \delta$ then $|s_n - t_n(p,q)| < \varepsilon$.
\end{lem}
\begin{proof}
	Suppose that the assertion is false so there exists a sequence $\{(p_k,q_k)\}_{k=1}^\infty$ of pairs of primes such that
	\begin{equation*}
		\lim_{k\to\infty} \frac{\log q_k}{\log p_k} = \phi\quad\mbox{and}\quad |s_n - t_n(p_k,q_k)| \geq \varepsilon.
	\end{equation*}
	For simplicity, we shall now write $r = s_n$ and $r_k = t_n(p_k,q_k)$ so we have that $|r - r_k| \geq \varepsilon$ for all $k\in \nat$.  By definition of $r_k$ we obtain that
	\begin{equation*}
		m\left( \frac{p_k^{h_n}}{q_k^{h_{n-1}}}\right) = \left(m\left( \frac{p_k^{h_{n-1}}}{q_k^{h_{n-2}}}\right)^{r_k}  + m\left( \frac{p_k^{h_{n-2}}}{q_k^{h_{n-3}}}\right)^{r_k}\right)^{1/r_k},
	\end{equation*}
	which simplifies to
	\begin{equation} \label{PQApproach}
		\max\left\{ h_n, \frac{\log q_k}{\log p_k} h_{n-1}\right\} = \left(\max\left\{ h_{n-1}, \frac{\log q_k}{\log p_k} h_{n-2}\right\}^{r_k} + \max\left\{ h_{n-2}, \frac{\log q_k}{\log p_k} h_{n-3}\right\}^{r_k}\right)^{1/r_k}.
	\end{equation}
	
	Before proceeding, we claim that $\{r_k\}$ is a bounded sequence.  By Proposition \ref{UniqueIntersection}, we have that $r_k \geq 1$, so it is sufficient to show that $\{r_k\}$ is bounded from above.  
	If $\{r_k\}$ is not bounded from above, there exists a subsequence $\{r_{k_i}\}$ such that $r_{k_i}\to \infty$ as $i \to \infty$.  Assuming first that $n$ is odd, we use the fact that
	$\lim_{k\to\infty} \log q_k/\log p_k = \phi$ to assume without loss of generality that
	\begin{equation*}
		\frac{h_{n+1}}{h_{n}} < \frac{\log q_{k_i}}{\log p_{k_i}} < \frac{h_{n+2}}{h_{n+1}} < \frac{h_{n}}{h_{n-1}}.
	\end{equation*}
	Now applying \eqref{PQApproach} we deduce that
	\begin{align*}
		h_n  = \max\left\{ h_n, \frac{\log q_{k_i}}{\log p_{k_i}} h_{n-1}\right\}
			 & \leq  \left(\left(  \frac{h_{n+2} h_{n-2}}{h_{n+1}}\right)^{r_{k_i}} + h_{n-2}^{r_{k_i}}\right)^{1/r_{k_i}}.
	\end{align*}
	We take the limit of both sides as $i\to \infty$ to obtain that
	\begin{equation*}
		h_n \leq  \lim_{i\to\infty} \left(\left(  \frac{h_{n+2} h_{n-2}}{h_{n+1}}\right)^{r_{k_i}} + h_{n-2}^{r_{k_i}}\right)^{1/r_{k_i}}
			= \lim_{t\to\infty} \left(\left(  \frac{h_{n+2} h_{n-2}}{h_{n+1}}\right)^{t} + h_{n-2}^{t}\right)^{1/t} = \frac{h_{n+2} h_{n-2}}{h_{n+1}} < \frac{h_n h_{n-2}}{h_{n-1}}
	\end{equation*}
	which leads to $h_{n-1} < h_{n-2}$, a contradiction.  In case $n$ is even, we assume that
	\begin{equation*}
		\frac{h_{n}}{h_{n-1}} < \frac{\log q_{k_i}}{\log p_{k_i}} < \frac{h_{n+1}}{h_{n}} < \frac{h_{n-1}}{h_{n-2}}.
	\end{equation*}
	Under these assumptions, we obtain that
	\begin{equation*}
		\frac{\log q_{k_i}}{\log p_{k_i}} h_{n-1} = \max\left\{ h_n, \frac{\log q_{k_i}}{\log p_{k_i}} h_{n-1}\right\} \leq 
			\left( h_{n-1}^{r_{k_i}} + \left( \frac{h_{n-1}h_{n-3}}{h_{n-2}}\right)^{r_{k_i}}\right)^{1/r_{k_i}}.
	\end{equation*}
	and taking limits of both sides as $i\to \infty$ yields $\phi h_{n-1} \leq h_{n-1}$ another contradiction.  Hence, we have now established that $\{r_k\}$ is bounded.
		
	By possibly replacing $\{r_k\}$ with a convergent subsequence, we may assume without loss of generality that $\lim_{k\to\infty} r_k = r'$.  Moreover, since $r_k\geq 1$ for all $k$ we know that $r'\geq 1$.
	Now define $f:(0,\infty)\times(0,\infty)\to \real$ by
	\begin{equation*}
		f(x,t) = \left( \max\{h_{n-1},xh_{n-2}\}^t + \max\{h_{n-2},xh_{n-3}\}^t \right)^{1/t}
	\end{equation*}
	so that $f$ is continuous at all points $(x,t)$ in its domain with respect to the usual Euclidean norm.  We observe that \eqref{PQApproach} now becomes
	\begin{equation} \label{PQApproach2}
		\max\left\{ h_n, \frac{\log q_k}{\log p_k} h_{n-1}\right\} = f\left( \frac{\log q_k}{\log p_k},r_k\right)
	\end{equation}
	and using the continuity of $f$ we get that
	\begin{equation} \label{Continuity}
		\lim_{k\to\infty} f\left( \frac{\log q_k}{\log p_k},r_k\right) = f\left(\lim_{k\to\infty}\left( \frac{\log q_k}{\log p_k},r_k\right)\right) = f(\phi,r').
	\end{equation}
	Taking limits of both sides of \eqref{PQApproach2} as $k\to \infty$ and applying \eqref{Continuity} we obtain that
	\begin{equation*}
		\max\left\{ h_n, \phi h_{n-1}\right\}  = \lim_{k\to\infty} \max\left\{ h_n, \frac{\log q_k}{\log p_k} h_{n-1}\right\} = \lim_{k\to\infty} f\left( \frac{\log q_k}{\log p_k},r_k\right) = f(\phi,r')
	\end{equation*}
	which is equivalent to
	\begin{equation*}
		\max\left\{ h_n, \phi h_{n-1}\right\} ^{r'} = \max\{h_{n-1},\phi h_{n-2}\}^{r'} + \max\{h_{n-2},\phi h_{n-3}\}^{r'}.
	\end{equation*}
	Now using the definition of $r = s_n$ and the uniqueness established in Lemma \ref{TargetUniqueIntersection}, we conclude that $r = r'$ so that
	$\lim_{k\to\infty} r_k = r$ contradicting our assumption that $|r_k - r| \geq \varepsilon$.
\end{proof}

With Lemmas \ref{TargetCompatible} and \ref{TApprox} we have finished the majority of the proof of Theorem \ref{IntersectionWeaving}.  We include the remainder of that proof now.

\begin{proof}[Proof of Theorem \ref{IntersectionWeaving}]
	Let $\varepsilon = \min\{(s_n - s_{n+1})/2: 3\leq n\leq N\}$ and note that Lemma \ref{TargetCompatible} implies that $\varepsilon > 0$.  From Lemma \ref{TApprox} there must exist $\delta > 0$ such that if 
	$|\log q/\log p - \phi| < \delta$ then $|s_n - t_n(p,q)| < \varepsilon$.  Assuming that $q$ and $p$ are such that $|\log q/\log p - \phi| < \delta$ we obtain that
	\begin{equation*}
		s_n - t_n(p,q) \leq  |s_n - t_n(p,q)| < \varepsilon < \frac{1}{2}(s_n - s_{n+1})
	\end{equation*}
	and
	\begin{equation*}
		t_{n+1}(p,q) - s_{n+1} \leq |s_{n+1} - t_{n+1}(p,q)| < \varepsilon < \frac{1}{2}(s_n - s_{n+1}).
	\end{equation*}
	Adding these inequalities we find that $s_n - t_n(p,q) + t_{n+1}(p,q) - s_{n+1} < s_n - s_{n+1}$ and the result follows.
\end{proof}

We find it worth noting that we are aware of a more direct proof of Theorem \ref{IntersectionWeaving}.  Specifically, there is a proof which does not require the use of the points $s_n$
or Lemmas \ref{TargetUniqueIntersection} and \ref{TargetCompatible}.  Nevertheless, we find the above proof to be more informative because it establishes not only that $t_n(p,q)$ satisfy the required
inequalities, but also that these points approach $s_n$ as $\log q/\log p \to \phi$.  As a result, in order to study the structure of the set $\{t_3(p,q),t_4(p,q),\ldots,t_{N}(p,q)\}$, it may instead be possible
to study $\{s_3,s_4,s_5,\ldots\}$.  Our above proof of Theorem \ref{IntersectionWeaving} is an example of this strategy.


\section{Proof of Theorem \ref{ConjectureConsequences}}

For the remainder of this article, we shall return to the notation utilized in Sections \ref{Intro}, \ref{ConjecturedReplace} and \ref{Progress}.  Specifically, we assume that $N\geq 3$ is an integer and that 
$\mathcal V_n$ and $\mathcal S_n$ are defined as above for $1\leq n\leq N$.   Additionally, we suppose that $p$ and $q$ are 
primes satisfying \eqref{WeakCompatible}, and if $3\leq n\leq N$, we let $t_n$ be the unique positive real number such that
\begin{equation*}
	m\left( \frac{p^{h_n}}{q^{h_{n-1}}}\right)^{t_n} = m\left( \frac{p^{h_{n-1}}}{q^{h_{n-2}}}\right)^{t_n}  + m\left( \frac{p^{h_{n-2}}}{q^{h_{n-3}}}\right)^{t_n}.
\end{equation*}
By Proposition \ref{UniqueIntersection} we know that $t_n\geq 1$.  Our proof of Theorem \ref{ConjectureConsequences} utilizes a result which describes $\min\{f_\xx(t):\xx\in \mathcal S_n\}$
as a piecewise function using measure functions.

\begin{thm} \label{MinimumOver}
	Suppose that  $N$ is a positive integer and $(p,q)$ is a pair of primes compatible with $N$.  If $3\leq n\leq N$ then
	\begin{equation*}
		\min\{f_\xx(t):\xx\in \mathcal S_n\} = \begin{cases} f_{\xx_n(3)}(t) & \mathrm{if}\ t_{3} \leq t \\
												f_{\xx_n(i)}(t) & \mathrm{if}\ t_{i} \leq t \leq t_{i-1}\ \mathrm{for\ some}\ 4\leq i\leq n \\
												f_{\xx_n(n+1)}(t) & \mathrm{if}\ t \leq t_{n}.
									\end{cases}
	\end{equation*}
	Moreover, if $t\not\in \{t_3,t_4,\ldots,t_n\}$ then there exists a unique point $\zz\in \mathcal S_n$ such that $f_{\zz}(t) = \min\{f_\xx(t):\xx\in \mathcal S_n\}$.
\end{thm}

Note that Theorem \ref{MinimumOver} excludes the cases where $n\in \{1,2\}$.  However, since $\mathcal S_1$ and $\mathcal S_2$ each contain one element, the behavior of 
$\min\{f_\xx(t):\xx\in \mathcal S_n\}$ is rather trivial in these cases.

The proof of Theorem \ref{MinimumOver} requires a lemma which describes the relationship between the points $t_i$ and the functions $f_{\xx_{n}(i)}(t)$.

\begin{lem} \label{VectorIntersection}
	If $3\leq i\leq n\leq N$ and $\xx_n(i)$ then the following conditions hold.
	\begin{enumerate}[(i)]
		\item $f_{\xx_n(i)}(t_i) = f_{\xx_{n}(i+1)}(t_i)$
		\item\label{Increasing} $f_{\xx_{n}(i)}(t) > f_{\xx_n(i+1)}(t)$ for all $t < t_i$
		\item\label{Decreasing} $f_{\xx_{n}(i)}(t) < f_{\xx_n(i+1)}(t)$ for all $t > t_i$
	\end{enumerate}
\end{lem}
\begin{proof}
	We first show that $t_i$ is the unique positive real number such that $f_{\xx_n(i+1)}(t_i) = f_{\xx_{n}(i)}(t_i)$.  Directly applying the definition of $\xx_n(i)$, we obtain that
	\begin{equation*}
		f_{\xx_n(i)}(t)^t = h_{n+1-i}m\left( \frac{p^{h_{i-2}}}{q^{h_{i-3}}} \right)^t + h_{n+2-i}m\left( \frac{p^{h_{i-1}}}{q^{h_{i-2}}} \right)^t.
	\end{equation*}
	Also using the definition of $\xx_n(i)$ we find that
	\begin{equation*}
		\xx_n(i+1) = (\ \underbrace{0,0,\ldots,0,0}_{i-2\mbox{ times}},h_{n-i},h_{n+1-i})
	\end{equation*}
	so that
	\begin{equation*}
		f_{\xx_{n}(i+1)}(t)^t = h_{n-i}m\left( \frac{p^{h_{i-1}}}{q^{h_{i-2}}} \right)^t + h_{n+1-i}m\left( \frac{p^{h_{i}}}{q^{h_{i-1}}} \right)^t.
	\end{equation*}
	We note immediately that the equality $f_{\xx_n(i)}(t) = f_{\xx_{n}(i+1)}(t)$ is equivalent to 
	\begin{equation*}
		h_{n+1-i}m\left( \frac{p^{h_{i-2}}}{q^{h_{i-3}}} \right)^t + h_{n+2-i}m\left( \frac{p^{h_{i-1}}}{q^{h_{i-2}}} \right)^t 
			= h_{n-i}m\left( \frac{p^{h_{i-1}}}{q^{h_{i-2}}} \right)^t + h_{n+1-i}m\left( \frac{p^{h_{i}}}{q^{h_{i-1}}} \right)^t
	\end{equation*}
	which simplifies to 
	\begin{equation*}
		 h_{n+1-i}m\left( \frac{p^{h_{i-2}}}{q^{h_{i-3}}} \right)^t + ( h_{n+2-i} - h_{n-i})m\left( \frac{p^{h_{i-1}}}{q^{h_{i-2}}} \right)^t = h_{n+1-i}m\left( \frac{p^{h_{i}}}{q^{h_{i-1}}} \right)^t
	\end{equation*}
	By using the recurrence relation from the Fibonacci sequence, we find this to be equivalent to
	\begin{equation*}
		h_{n+1-i}m\left( \frac{p^{h_{i-2}}}{q^{h_{i-3}}} \right)^t + h_{n+1-i}m\left( \frac{p^{h_{i-1}}}{q^{h_{i-2}}} \right)^t = h_{n+1-i}m\left( \frac{p^{h_{i}}}{q^{h_{i-1}}} \right)^t
	\end{equation*}
	Since we have assumed that $i\leq n$, we know that $h_{n+1-i} > 0$.  Hence, we have shown that
	\begin{equation*}
		f_{\xx_n(i)}(t) = f_{\xx_{n}(i+1)}(t) \iff m\left( \frac{p^{h_{i-2}}}{q^{h_{i-3}}} \right)^t + m\left( \frac{p^{h_{i-1}}}{q^{h_{i-2}}} \right)^t = m\left( \frac{p^{h_{i}}}{q^{h_{i-1}}} \right)^t,
	\end{equation*}
	and it follows from Lemma \ref{UniqueIntersection} that $t_i$ is the unique positive real number such that $f_{\xx_n(i)}(t_i) = f_{\xx_{n}(i+1)}(t_i)$.
	
	By a similar argument, we also obtain that
	\begin{equation*}
		f_{\xx_n(i)}(t) > f_{\xx_{n}(i+1)}(t) \iff \left(m\left( \frac{p^{h_{i-2}}}{q^{h_{i-3}}} \right)^t + m\left( \frac{p^{h_{i-1}}}{q^{h_{i-2}}} \right)^t\right)^{1/t} > m\left( \frac{p^{h_{i}}}{q^{h_{i-1}}} \right),
	\end{equation*}
	and
	\begin{equation*}
		f_{\xx_n(i)}(t) < f_{\xx_{n}(i+1)}(t) \iff \left(m\left( \frac{p^{h_{i-2}}}{q^{h_{i-3}}} \right)^t + m\left( \frac{p^{h_{i-1}}}{q^{h_{i-2}}} \right)^t\right)^{1/t} < m\left( \frac{p^{h_{i}}}{q^{h_{i-1}}} \right).
	\end{equation*}
	In both of the right hand inequalities, the expression on the right is constant and the expression on the left is strictly decreasing as a function of $t$.  The remaining statements of the lemma now
	follow immediately.
\end{proof}

Equipped with the previous lemma, we are ready to prove Theorem \ref{MinimumOver}.

\begin{proof}[Proof of Theorem \ref{MinimumOver}]
	We first suppose that $t \geq t_3$ so that the definition of combatible means that $t_n < t_{n-1} < \cdots < t_4 < t_3 \leq t$.
	Hence, we may apply Lemma \ref{VectorIntersection}\eqref{Decreasing} to conclude that
	\begin{equation} \label{IQChain1}
		f_{\xx_{n}(3)}(t) \leq f_{\xx_{n}(4)}(t) < \cdots < f_{\xx_{n}(n)}(t) < f_{\xx_{n}(n+1)}(t)
	\end{equation}
	and the result follows.  On the other hand, if we consider the case where $t \leq t_n$ then we have $t \leq t_n < t_{n-1} < \cdots < t_4 < t_3$ and we apply 
	Lemma \ref{VectorIntersection}\eqref{Increasing} to conclude that 
	\begin{equation} \label{IQChain2}
		f_{\xx_{n}(n+1)}(t) \leq f_{\xx_{n}(n)}(t) < \cdots < f_{\xx_{n}(4)}(t)  < f_{\xx_{n}(3)}(t) 
	\end{equation}
	and the result follows in this case as well.  Finally, we suppose that $4 \leq i\leq n$ and that $t_i \leq t \leq t_{i-1}$ which means that
	\begin{equation*}
		t_n < t_ {n-1} < \cdots t_{i+1} < t_i \leq t \leq t_{i-1} < t_{i-1} < \cdots < t_4 < t_3.
	\end{equation*}
	Then applying Lemma \ref{VectorIntersection}\eqref{Decreasing} we get that
	\begin{equation} \label{IQChain3}
		f_{\xx_n(i)}(t) \leq  f_{\xx_n(i+1)}(t) < \cdots < f_{\xx_n(n)}(t) < f_{\xx_n(n+1)}(t),
	\end{equation}
	and by applying Lemma \ref{VectorIntersection}\eqref{Increasing} we obtain that
	\begin{equation} \label{IQChain4}
		f_{\xx_n(i)}(t) \leq f_{\xx_n(i-1)}(t) < \cdots < f_{\xx_n(4)}(t) < f_{\xx_n(3)}(t) 
	\end{equation}
	so the result follows in this case as well.  The final assertion of the theorem follows from the fact that the inequalities in \eqref{IQChain1}, \eqref{IQChain2}, \eqref{IQChain3}
	and \eqref{IQChain4} are all strict inequalities when $t\not\in \{t_3,t_4,\ldots,t_n\}$.
\end{proof}

Theorem \ref{MinimumOver} constitutes the majority of the proof of Theorem \ref{ConjectureConsequences}, however, there are some additional details that need to be sorted out.

\begin{lem}\label{MeasureFunctionsSwitch}
	Suppose that $1\leq n\leq N$ and that $\xx,\yy\in \mathcal V_n$.  If $f_\xx(t) = f_\yy(t)$ for infinitely many values of $t\in (0,\infty)$ then $\xx = \yy$.
\end{lem}
\begin{proof}
	Suppose that $\xx = (x_1,x_2,\ldots,x_N)$ and $\yy = (y_1,y_2,\ldots,y_N)$ so that
	\begin{equation*}
		f_\xx(t) =  \left(\sum_{i=1}^N x_i m\left(\frac{p^{h_i}}{q^{h_{i-1}}}\right)^t\right)^{1/t} \quad\mbox{and}\quad 
			f_\yy(t) =  \left(\sum_{i=1}^N y_i m\left(\frac{p^{h_i}}{q^{h_{i-1}}}\right)^t\right)^{1/t}.
	\end{equation*}
	Supposing that $\xx \ne \yy$ then we may assume that $j$ is the largest integer such that $x_j \ne y_j$, and we assume without loss of generality that $x_j > y_j$.
	Consequently, we find that
	\begin{equation*}
		\lim_{t\to \infty} \left(f_\xx(t)^t - f_\yy(t)^t\right)^{1/t} = m\left( \frac{p^{h_j}}{q^{h_{j-1}}}\right) > 0,
	\end{equation*}
	and therefore, $f_\xx(t) > f_\yy(t)$ for all sufficiently large $t$.  We have now established the existence of a closed interval $I\subseteq [0,\infty)$ such that $f_\xx(t) = f_\yy(t)$ for infinitely many values
	of $t\in I$.  Since $t\mapsto f_\xx(t)^t$ and $t\mapsto f_\yy(t)^t$ define entire functions for $t\in \com$, we conclude that $f_\xx(t) = f_\yy(t)$ for all $t\in \com$, a contradiction.
\end{proof}

All of our previous lemmas enable the proof of Theorem \ref{ConjectureConsequences}.

\begin{proof}[Proof of Theorem \ref{ConjectureConsequences}]
	To prove the first assertion, we assume that $\mathcal X\subseteq \mathcal V_n$ is such that $m_t(\alpha_n) = \min\{f_\xx(t):\xx\in \mathcal X\}$.
	Of course, we also assume that $m_t(\alpha_n) = \min\{f_\xx(t):\xx\in \mathcal S_n\}$.  It is easily checked that $\#\mathcal S_1 = \#\mathcal V_1 = 1$
	and $\#\mathcal S_2 = \#\mathcal V_2 = 1$ so the assertion is trivial in the cases of $n\in \{1,2\}$.  Therefore, we may assume that $3\leq n\leq N$ and are permitted to apply Theorem \ref{MinimumOver}.
	
	We must now prove that $\xx_n(i)\in\mathcal X$ for all $3\leq i\leq n+1$.  First assuming that $4\leq i\leq n$, Theorem \ref{MinimumOver} implies that $m_t(\alpha_n) = f_{\xx_n(i)}(t)$ for
	all $t\in [t_{i},t_{i-1}]$.  By our assumptions, for each $t\in [t_{i},t_{i-1}]$ there must exist $\yy_t\in \mathcal X$ such that $f_{\yy_t}(t) = f_{\xx_n(i)}(t)$.  Since $\mathcal X$ is certainly finite, 
	the pigeonhole principle implies the existence of $\yy\in \mathcal X$ such that $f_\yy(t) = f_{\xx_n(i)}(t)$ for infinitely many values of $t\in (0,\infty)$.  Then Lemma \ref{MeasureFunctionsSwitch}
	yields that $\xx_n(i) = \yy\in \mathcal X$ as required.  A similar argument applies in the cases where $ i =3$ and $i=n+1$ which completes the proof of the first assertion.
	
	We must now prove that $\{t_3,t_4,\ldots,t_n\}$ is the precise set of exceptional points for $\alpha_n$.  It clearly follows from Theorem \ref{MinimumOver} that all points outside of this set
	are standard.  If $t_i$ is standard for $\alpha_n$ then there exists $\bar\alpha\in \mathcal P(\alpha)$ and $0 < \varepsilon < \min\{t_i - t_{i+1}:3\leq i\leq n\}$ such that $m_t(\alpha_n) = f_{\bar\alpha}(t)$ for all 
	$t\in (t_i-\varepsilon,t_i+\varepsilon)$.  Now it follows from Theorem \ref{MinimumOver} that
	\begin{equation*}
		f_{\bar\alpha}(t) = f_{\xx_n(i+1)}(t)\quad\mbox{for all } t\in (t_i-\varepsilon,t_i)
	\end{equation*}
	and
	\begin{equation*}
		f_{\bar\alpha}(t) = f_{\xx_n(i)}(t)\quad\mbox{for all } t\in (t_i,t_i+\varepsilon).
	\end{equation*}
	Since $t\mapsto f_{\bar\alpha}(t)^t$, $t\mapsto f_{\xx_n(i+1)}(t)^t$ and $t\mapsto f_{\xx_n(i)}(t)^t$ define analytic functions for $t\in \com$, we conclude that they are all equal for $t\in \com$.
	This forces $f_{\xx_n(i+1)}(t) = f_{\xx_n(i)}(t)$ for all $t\in (0,\infty)$, and Lemma \ref{MeasureFunctionsSwitch} implies that $\xx_n(i+1) = \xx_n(i)$, a contradiction.
\end{proof}


\section{Proof of Theorem \ref{MainProgress}} \label{MainProgressProofs}

The proof of the set containments \eqref{ProgressContainments} in Theorem \ref{MainProgress} requires the following linear algebra lemma.

\begin{lem} \label{SEquals}
	Suppose that $1\leq i \leq N$ and that $\zz = (z_1,z_2,\ldots,z_N)\in \mathcal V_i$.  If there exists $2\leq j\leq i$ such that $z_\ell = 0$ for all $\ell\not\in \{j-1,j\}$ then $\zz = \xx_i(j+1)$.
	In particular, $\zz\in \mathcal S_i$.
\end{lem}
\begin{proof}
	Since $\zz, \xx_i(j+1)\in \mathcal V_i$, we immediately notice that $A(\zz - \xx_i(j+1)) = {\bf 0}$.  However, both vectors $\zz$ and $\xx_i(j+1)$ have $0$'s in every entry except possibly in entries $j-1$ and $j$.
	Therefore, it follows that
	\begin{equation} \label{MatrixEquals}
		\begin{pmatrix} h_{j-1} & h_j \\ h_{j-2} & h_{j-1} \end{pmatrix}\left( \begin{pmatrix} z_{j-1} \\ z_{j} \end{pmatrix} - \begin{pmatrix} h_{i-j} \\ h_{i-j+1} \end{pmatrix}\right) = {\bf 0}
	\end{equation}
	It is a straightforward proof by induction on $j$ that the $2\times 2$ matrix on the left hand side of \eqref{MatrixEquals} has determinant equal to $(-1)^{j}$.  
	The result now follows by multiplying both sides of \eqref{MatrixEquals} by the inverse of this matrix.
\end{proof}

The proof of the set containments in \eqref{ProgressContainments} at the beginning of Theorem \ref{MainProgress} can be done immediately, so we include that proof here.  
The remainder of the proof of Theorem \ref{MainProgress} is provided later in this section.

\begin{proof}[Proof of \eqref{ProgressContainments}]
	It follows directly from the definitions that $\mathcal R_n \subseteq \mathcal C_n \subseteq  \mathcal V_n$, so it remains only to show that $\mathcal S_n \subseteq \mathcal R_n$.
	Suppose that $\zz\in \mathcal S_n$ so we know that $\zz$ must have the form 
	\begin{equation*}
		\zz = ( \underbrace{0,0,\ldots,0,0}_{j-2\mbox{ times}},z_{j-1},z_{j})
	\end{equation*}
	for some $1\leq j\leq n$ (In the case $j=1$, our notation should be interpreted as $\zz = (z_j)$).  Further suppose that $(\yy_1,\yy_2,\ldots,\yy_K)$ is factorization of $\zz$ and fix $k\in\intg$ such that $1\leq k\leq K$.  
	We must show that there exists $i$ such that $\yy_k\in \mathcal S_i$.  To see this, we first observe that $\yy_k$ must have a zero in every every except possibly in the $(j-1)$th and $j$th entries so we may write
	\begin{equation} \label{YVector}
		\yy_k = ( \underbrace{0,0,\ldots,0,0}_{j-2\mbox{ times}},y_{j-1},y_{j}).
	\end{equation} 
	Moreover, we know from the definition of factorization that $\yy_k\in \mathcal V_i$ for some $i$, and note that we must have $ j \leq i+1$ because otherwise $\yy_k = {\bf 0}\not\in \mathcal V_i$.
	Now we consider the following three cases.
	\begin{enumerate}
		\item If $j = 1$ then $A\yy_k = (y_j,0)$ forcing $i =1$ and $y_j= 1$.  This means that $\yy_k = \xx_i(2)$.
		\item If $2\leq j\leq i$ then Lemma \ref{SEquals} yields that $\yy_k = \xx_i(j+1)$.
		\item If $j = i+1$ then $y_j = 0$ and $y_{j-1} = 1$ so that $\yy_k = \xx_i(j)$.
	\end{enumerate}
	In all cases, we observe that $\yy\in \mathcal S_i$ as required.
\end{proof}

In order to complete the proof of Theorem \ref{MainProgress}, we shall require several additional definitions and preliminary lemmas.
We once again remind the reader that all definitions depend on the choices of $N$, $p$ and $q$ even though our notation will not reflect these dependencies.
For a point $\xx\in \mathcal V_n$, we define define the {\it infimum attaining set for $\xx$} to be $$\mathfrak U_\xx = \{t\in (0,\infty): m_t(\alpha_n) = f_\xx(t)\}.$$  The following lemma establishes
that we may essentially disregard any point $\xx\in \mathcal V_n$ for which $\mathfrak U_\xx$ is finite.

\begin{lem} \label{InfAttainSet}
	Suppose that $\mathcal X \subseteq \mathcal Y \subseteq \mathcal V_n$ and that  $\mathfrak U_\zz$ is finite for all $\zz \in\mathcal Y\setminus \mathcal X$.  If 
	$m_t(\alpha_n) = \min\{f_\xx(t): \xx\in \mathcal Y\}$ then $m_t(\alpha_n) = \min\{f_\xx(t): \xx\in \mathcal X\}$.
\end{lem}
\begin{proof}
	Suppose that $\yy\in \mathcal Y$ and that $s \in (0,\infty)$ are such that $m_s(\alpha_n) = f_\yy(s)$.  We need to show that there exists $\xx\in \mathcal X$ such that $m_s(\alpha_n) = f_\xx(s)$.
	To see this, let $\{s_i\}_{i=1}^\infty$ be a sequence of distinct points in $(0,\infty)$ converging to $s$.  For each $i$, there exists a point $\xx_i\in \mathcal Y$ such that $m_{s_i}(\alpha_n) = f_{\xx_i}(s_i)$.
	Since $\mathcal Y$ is finite, we apply the Pigeonhole Principle to assume without loss of generality that there exists $\xx\in \mathcal Y$ such that $m_{s_i}(\alpha_n) = f_\xx(s_i)$ for all $i$.
	This assertion clearly means that $\mathfrak U_\xx$ is infinite, so we conclude that $\xx\in \mathcal X$.  Moreover, by continuity of the maps $t\mapsto f_\xx(t)$ and $t\mapsto m_t(\alpha_n)$ 
	we deduce that
	\begin{equation*}
		f_\xx(s) = \lim_{i\to\infty} f_\xx(s_i) = \lim_{i\to\infty} m_{s_i}(\alpha_n) = m_s(\alpha_n) 
	\end{equation*}
	which completes the proof of the lemma.
\end{proof}

Our next lemma gives us a strategy to prove that $\mathfrak U_\zz$ is finite given a particular factorization of $\zz$.

\begin{lem} \label{FactorizationLemma}
	If $\zz\in \mathcal V_n$ and $(\xx_1,\xx_2,\ldots,\xx_K)$ is a factorization of $\zz$ then
	\begin{equation} \label{FactorizationContainment}
		\mathfrak U_\zz\subseteq \bigcap_{k=1}^K \mathfrak U_{\xx_k}.
	\end{equation}
\end{lem}
\begin{proof}
	Suppose that $t\in \mathfrak U_\zz$ and assume without loss of generality that $t\not\in \mathfrak U_{\xx_1}$.  Additionally, we may assume that $\xx_1 \in \mathcal V_j$ for some $1\leq j\leq n$.
	Since $t\not\in \mathfrak U_{\xx_1}$ there must exist $\yy\in \mathcal V_j$ such that $f_{\xx_1}(t) > f_{\yy}(t)$.  Then using the linearity of the map $\xx\mapsto f_\xx(t)^t$ and setting 
	$\zz' = \yy + \xx_2 + \cdots + \xx_K$  we obtain that
	\begin{equation*}
		f_\zz(t)^t = \sum_{k=1}^K f_{\xx_k}(t)^t > f_{\yy}(t) + \sum_{k=2}^K f_{\xx_k}(t)^t = f_{\zz'}(t)^t.
	\end{equation*}
	Moreover, $\zz'$ certainly has non-negative integer entries and $A(\zz') = A(\zz)$.  This implies that $\zz'\in \mathcal V_n$ so we contradict the fact that $t\in \mathfrak U_\zz$.
\end{proof}

The combination of Lemmas \ref{InfAttainSet} and \ref{FactorizationLemma} suggests a strategy to prove Conjecture \ref{MainConj}.  If $\zz = (z_1,\ldots,z_n) \in \mathcal V_n\setminus \mathcal S_n$ has a 
factorization satisfying the hypotheses of Lemma \ref{FactorizationLemma}, then we can show that $\mathfrak U_\zz$ is finite by showing that $\cap_{k=1}^K \mathfrak U_{\xx_k}$ is finite.  
Then we may apply Lemma \ref{InfAttainSet} to eliminate each such point $\zz$ from consideration in $m_t(\alpha_n)$.  

The main advantage of this approach is that it is often easier to study $\mathfrak U_{\xx_k}$ than it is to study $\mathfrak U_\zz$. 
For example, if we are attempting to prove Conjecture \ref{MainConj} by induction on $n$, then any non-trivial factorization $(\xx_1,\xx_2,\ldots,\xx_K)$
will have $\xx_k\in \mathcal V_{i}$ for some $1\leq i < n$.  Therefore, we would have the inductive hypothesis that
\begin{equation*}
	m_t(\alpha_i) = \min\{f_\xx(t): \xx\in \mathcal S_i\}\quad\mbox{for all } 1\leq i < n
\end{equation*}
to assist us in showing that $\cap_{k=1}^K \mathfrak U_{\xx_k}$ is finite.  

The primary disadvantage of the above strategy is that it cannot be used with the factorizations defined by \eqref{TrivialFactorizations}.  
Indeed, the former has $K=1$ and $\zz = \xx_1$ so Lemma \ref{FactorizationLemma} provides no information, and the latter satisfies
\begin{equation*}
	\bigcap_{k=1}^K \mathfrak U_{\xx_k} = (0,t_j],
\end{equation*}
where $j$ is the largest index such that $z_j \ne 0$.  Hence, we would need to obtain a more creative factorization for $\zz$ than those appearing in 
\eqref{TrivialFactorizations}, and unfortunately, such a factorization does not always exist.  For example, it can be shown that the only factorizations of $(1,0,0,4)\in \mathcal V_7$ are the trivial
and improper factorizations.   As a result, this strategy cannot be used to provide a complete proof of Conjecture \ref{MainConj},
although it does lead to the remainder of our proof of Theorem \ref{MainProgress}.

\begin{proof}[Remainder of the Proof of Theorem \ref{MainProgress}]
	To prove \eqref{Consecutive}, it is sufficient to assume that $\zz\in \mathcal V_n\setminus \mathcal C_n$ and prove that $\mathfrak U_\zz$ is finite.  Indeed, then Lemma \ref{InfAttainSet} would
	imply the desired result.  To see this, we assume that $\zz = (z_1,z_2,\ldots,z_N)\in \mathcal V_n$ and there exists $1\leq j < N-1$ satisfying the following properties:
	\begin{enumerate}
		\item $z_j \ne 0$ and $z_{j+1} \ne 0$
		\item There exists $k > j+1$ such that $z_k \ne 0$.
	\end{enumerate}
	By our assumptions, we surely have that $j+1 < k < n$ and we write
	\begin{equation*}
		z'_i = \begin{cases} z_i & \mbox{if } i \not \in\{j,j+1,k\} \\ z_i - 1 &\mbox{if } i \in \{j,j+1,k\}. \end{cases}
	\end{equation*}
	From our assumptions we know that $z'_i \geq 0$ for all $1\leq i\leq N$.  Now we observe that
	\begin{align*}
		(z_1,z_2,\ldots, z_N) & =  ( \underbrace{0,0,\ldots,0,0}_{j-1\mbox{ times}},1,1,\underbrace{0,0,\ldots,0,0}_{N-j-1\mbox{ times}}) +
			 ( \underbrace{0,0,\ldots,0,0}_{k-1\mbox{ times}},1,\underbrace{0,0,\ldots,0,0}_{N-k\mbox{ times}}) + (z'_1,z'_2,\ldots,z'_N) \\
			 & = \xx_{j+2}(j+2) + \xx_k(k+1) + \sum_{i=1}^N z'_i\xx_i(i+1)
	\end{align*}
	However, using Lemma \ref{VectorIntersection}, we conclude that 
	\begin{equation*}
		\mathfrak U_{\xx_{j+2}(j+2)} \subseteq [t_{j+2},t_{j+1}] \quad\mbox{and}\quad \mathfrak U_{\xx_k(k+1)} \subseteq (0,t_k].
	\end{equation*}
	We also know that $k\geq j+2$ so these intervals have at most one point in common.  It now follows from Lemma \ref{FactorizationLemma} that $\mathfrak U_\zz$ is finite establishing \eqref{Consecutive}.
	
	We now complete the proof by establishing \eqref{Restricted}.  From \eqref{Consecutive}, we know that $m_t(\alpha_n) = \min\{f_\xx(t):\xx\in\mathcal C_n\}$ so we assume that 
	$\zz\in \mathcal C_n\setminus \mathcal R_n$.  Again, it is sufficient to show that $\mathfrak U_\zz$ is finite.
	Since $\zz\not\in \mathcal R_n$, there must exist a nontrivial factorization $(\xx_1,\xx_2,\ldots,\xx_K)$ which is not $\mathcal S$-restricted.  We assume without loss of generality that 
	$\xx_1\in \mathcal V_i\setminus \mathcal S_i$, and since $\zz\in \mathcal C_n$, we may also assume that $\xx_1\in \mathcal C_i$.  Because $(\xx_1,\xx_2,\ldots,\xx_K)$ is a non-trivial
	factorization, we know that $1\leq i < n$, and therefore our assumption yields
	\begin{equation*}
		f_{\xx_1}(t) \geq m_t(\alpha_i) = \min\{f_\xx(t):\xx\in \mathcal S_i\}\quad \mbox{for all } t > 0.
	\end{equation*}
	If $\mathfrak U_{\xx_1}$ is infinite, then by the Pigeonhole Principle, there exists $\yy\in \mathcal S_i$ such that $f_{\xx_1}(t) = f_{\yy}(t)$ for infinitely many values of $t$.
	Lemma \ref{MeasureFunctionsSwitch} implies that $\xx_1 = \yy$ contradicting our assumption that $\xx_1\not\in \mathcal S_i$.
\end{proof}

Our last remaining task is to prove Lemma \ref{RestrictedComputations}.

\begin{proof}[Proof of Lemma \ref{RestrictedComputations}]
	We first assume that there $\xx\in \mathcal R_i\setminus \mathcal S_i$ is such that all entries of $\zz - \xx$ are non-negative and write $\zz - \xx = (y_1,y_2,\ldots,y_N)$, where $y_i \geq 0$.
	Therefore, we conclude that
	\begin{equation*}
		\zz = \xx + (\zz - \xx) = \xx + \sum_{i=1}^N y_i\xx_i(i+1).
	\end{equation*}
	We cannot have $\xx = \zz$ because then $\xx \in \mathcal R_n$ contradicting our assumption that $\xx \in \mathcal R_i$ for $1\leq i < n$.  Additionally, we have assumed that
	$\xx \not\in \mathcal S_i$ so that we have identified a factorization of $\zz$ which is not $\mathcal S$-type, as required.
	
	To prove the other statement we must show that $\zz \in \mathcal C_n\setminus \mathcal R_n$ implies that there exists $1\leq i < n$ and 
	$\xx\in \mathcal R_i\setminus \mathcal S_i$ such that the entires of $\zz - \xx$ are non-negative.  We shall prove this assertion using induction on $n$ and we note that the base case is vacuously correct
	since $\mathcal C_1\setminus \mathcal R_1 = \emptyset$.  For the inductive step, we assume that for every $k < n$, $\yy \in \mathcal C_k\setminus \mathcal R_k$ implies that there exists $1\leq i < k$ and 
	$\xx\in \mathcal R_i\setminus \mathcal S_i$ such that the entires of $\yy - \xx$ are non-negative.
	
	Assuming that $\zz\in \mathcal C_n\setminus \mathcal R_n$ we know that $\zz$ must have a non-trivial factorization which is not $\mathcal S$-type.  Therefore, we may write
	\begin{equation} \label{ZFactor}
		\zz = \yy + \sum_{\ell=1}^L \xx_\ell,
	\end{equation}
	where $\yy \in \mathcal C_k\setminus \mathcal S_k$ and $\xx_\ell\in \cup_{i=1}^{n-1}\mathcal C_i$.  
	If $\yy\in \mathcal R_k$ then we use $\xx = \yy$ and $i = k$ to conclude the desired result.  If $\yy\not\in \mathcal R_k$ then the inductive hypothesis identifies a value $j$, with $1\leq j < k$,
	and $\xx\in \mathcal R_j\setminus\mathcal S_j$ such that $\yy - \xx$ has only non-negative entries.  Using these observations in \eqref{ZFactor}, we obtain that
	\begin{equation*}
		\zz - \xx = \yy - \xx + \sum_{\ell=1}^L \xx_\ell
	\end{equation*}
	must have only non-negative entires, completing the proof.

\end{proof}

By making a basic linear programming observation, it is possible to make an additional minor improvement to the results of Theorem \ref{MainProgress}.
A point $\zz\in \mathcal R_n$ is called a {\it vertex of $\mathcal R_n$} if $\zz$ cannot be written as a convex combination of the other points in $\mathcal R_n$.
Since the map $\xx\mapsto f_\xx(t)^t$ is a linear map, it is well-known that $\min\{f_\xx(t):\xx\in \mathcal R_n\}$ must be attained at a vertex of $\mathcal R_n$.
In conjunction with Lemma \ref{MeasureFunctionsSwitch}, these observations imply that each non-vertex of $\mathcal R_n$ may only attain the infimum in $m_t(\alpha_n)$
at finitely many points.  As a result, Lemma \ref{InfAttainSet} enables us to remove all such points from consideration.

As an example, take $n=11$ and note that some points in $\mathcal R_{11}$ are not vertices of $\mathcal R_{11}$.  Specifically, we observe that
\begin{equation*}
	\begin{pmatrix} 1 \\ 0 \\ 0 \\ 1 \\0 \\8 \\0 \\1\end{pmatrix} = 
		\frac{2}{3}\cdot \begin{pmatrix} 1 \\ 0 \\ 0 \\0 \\ 0 \\11 \\ 0 \\ 0 \end{pmatrix} + \frac{1}{3} \cdot \begin{pmatrix} 1 \\ 0 \\ 0 \\ 3 \\0 \\2 \\0 \\3\end{pmatrix}
		\quad\mbox{and}\quad
	\begin{pmatrix} 1 \\ 0 \\ 0 \\ 2 \\0 \\5 \\0 \\2\end{pmatrix} = 
		\frac{1}{3}\cdot \begin{pmatrix} 1 \\ 0 \\ 0 \\0 \\ 0 \\11 \\ 0 \\ 0 \end{pmatrix} + \frac{2}{3} \cdot \begin{pmatrix} 1 \\ 0 \\ 0 \\ 3 \\0 \\2 \\0 \\3\end{pmatrix}
\end{equation*}
meaning that $(1,0,0,1,0,8,0,1)$ and $(1,0,0,2,0,5,0,2)$ are not vertices of $\mathcal R_{11}$.  Consequently, we shall set
\begin{equation*}
	\mathcal T_{11} = \left\{ \begin{pmatrix} 0 \\ 0 \\ 0 \\0 \\ 1 \\0 \\ 0 \\ 4 \end{pmatrix},  \begin{pmatrix} 0 \\ 0 \\ 1 \\0 \\ 0 \\3 \\ 0 \\ 3 \end{pmatrix}, 
				\begin{pmatrix} 0 \\ 1 \\ 0 \\0 \\ 2 \\0 \\ 6 \\ 0 \end{pmatrix},  \begin{pmatrix} 1 \\ 0 \\ 0 \\0 \\ 0 \\11 \\ 0 \\ 0 \end{pmatrix},
				 \begin{pmatrix} 1 \\ 0 \\ 0 \\1 \\ 0 \\9 \\ 1 \\ 0 \end{pmatrix}, \begin{pmatrix} 1 \\ 0 \\ 0 \\ 3 \\0 \\2 \\0 \\3\end{pmatrix}\right\}.
\end{equation*}
Assuming we already know that $m_t(\alpha_i) = \min\{f_\xx(t):\xx\in \mathcal S_i\}$ for all $1\leq i < 11$, then we may conclude that 
$m_t(\alpha_{11}) = \min\{f_\xx(t):\xx\in \mathcal T_{11}\cup \mathcal S_{11}\}$, a slight improvement over applying Theorem \ref{MainProgress}\eqref{Restricted} directly in this case.

This strategy is not sufficient to prove Conjecture \ref{MainConj} as not all points in $\mathcal R_n\setminus \mathcal S_n$ may be written as a convex combination of points in $\mathcal S_n$.
Moreover, we are unaware of an efficient computational method for determining the precise list of vertices of $\mathcal R_n$.  Hence, we don't believe that these observations alone
contribute significantly to our work in this article.  Nevertheless, they do provide some hope that Conjecture \ref{MainConj} could be solved using one of the well-known linear programming
techniques (see \cite{Bert,Sierksma} for a discussion of these methods).

\end{document}